\documentclass[11pt,letterpaper]{article}

\usepackage[affil-it]{authblk}

\usepackage[utf8]{inputenc}
\usepackage{amsmath}
\usepackage{amsfonts}
\usepackage{amssymb}
\usepackage{amsthm}
\usepackage{graphicx}
\usepackage{subcaption} 
\usepackage[left=2cm,right=2cm,top=2cm,bottom=2cm]{geometry}
\usepackage{enumerate}
\usepackage{appendix}
\usepackage{mathtools}
\usepackage{xcolor}
\usepackage{hyperref}
\usepackage{todonotes}
\usepackage[english]{babel}
\usepackage{tikz-cd}

\newtheorem{theorem}{Theorem}[section]

\newtheorem{prop}[theorem]{Proposition}

\theoremstyle{definition}
\newtheorem{definition}[theorem]{Definition}

\theoremstyle{remark}
\newtheorem{remark}[theorem]{Remark}

\DeclareMathOperator{\dist}{dist}

\newcommand{\pr}{\text{pr}}

\numberwithin{equation}{section}

\usepackage[maxbibnames=99,backend=bibtex]{biblatex}
\begin{document}
\title{\bf Fast Reactions and Slow Manifolds} 
\date{}
\author{Christian Kuehn$\,^1$, Jan-Eric Sulzbach$\,^2$}
\affil{\small
    $\,^1$ Department of Mathematics, Technical University of Munich\\
    email: ckuehn@ma.tum.de\\
    
    \vspace{0.3cm}
    $\,^2$ Department of Mathematics, Technical University of Munich\\
    email: janeric.sulzbach@ma.tum.de\\
    
}

\maketitle

\begin{abstract}
In this paper we generalize the Fenichel theory for attracting critical/slow manifolds to fast-reaction systems in infinite dimensions. In particular, we generalize the theory of invariant manifolds for fast-slow partial differential equations in standard form to the case of fast reaction terms. We show that the solution of the fast-reaction system can be approximated by the corresponding slow flow of the limit system. Introducing an additional parameter that stems from a splitting in the slow variable space, we construct a family of slow manifolds and we prove that the slow manifolds are close to the critical manifold. Moreover, the semi-flow on the slow manifold converges to the semi-flow on the critical manifold. Finally, we apply these results to an example and show that the underlying assumptions can be verified in a straightforward way.
\end{abstract}

\section{Introduction}

The study of fast-reaction systems plays a crucial role in many natural sciences. Examples of fast-reaction systems and their limit can be found in reaction-diffusion models \cite{murakawa2011fast, fellner2016quasi, henneke2016fast}, enzyme reactions \cite{doi:10.1063/1.1889434}, coagulation-fragmentation models \cite{doi:10.1081/PDE-120021188,carrillo2008fast} and many other bio-chemical systems \cite{guarguaglini2007fast,herrmann2014rate,de2019fast}. The important aspect of fast-reaction limits is that they can help to reduce the complexity of a system by eliminating the fast evolving dynamics of the system and only the evolution of the slow dynamics remains. One approach, that stems from the modeling of chemical reactions, is to regard the fast-reaction limit as a quasi-steady-state approximation (QSSA) of the system \cite{zbMATH06088800, eilertsen2021hunting}. The mathematical study of fast-reactions in the context of reaction-diffusion dates back at least to the seminal work by Hilhorst et al. \cite{hilhorst1996fast}. A more recent review of the topic can be found in \cite{iida2018review}; we also refer to~\cite{crooks2007fast,doi:10.1080/03605300903225396,hausberg2018well,daus2020cross} for several further examples analyzing fast-reaction systems.

One important question that arises in the study of fast-reaction systems is to rigorously prove the convergence of the system and its solutions towards the fast-reaction limit. Different approaches such as energy and entropy methods \cite{Bothe201, skrzeczkowski2022fast}, variational methods \cite{disser2018evolutionary,mielke2021edp, stephan2021edp} or homogenization techniques \cite{meier2010two,cesaroni2017homogenization} have been proposed in recent years. Our approach follows a recent geometric viewpoint using dynamical systems methods as introduced in \cite{hummel2022slow} for fast-slow systems. Here, the idea is to generalize the theory developed by Fenichel \cite{fenichel1971persistence,fenichel1979geometric} for ODE systems and combine it with techniques for infinite-dimensional Banach spaces cf.~\cite{chow1988invariant,debussche1991inertial, henry2006geometric} and the references therein. The challenges and also several applications of this approach can be found in \cite{engel2020blow,engel2021connecting, hummel2022slow}. Yet, the case of fast-reaction system is not covered by the results in~\cite{hummel2022slow}, which is the main focus of this work.
 % and here we cover this case. 
More precisely, the fast-reaction evolution equations analyzed in this work are the following
\begin{align}\label{fast-reaction}
    \begin{split}
        \partial_t u^\varepsilon &=A u^\varepsilon +\frac{1}{\varepsilon} f(u^\varepsilon, v^\varepsilon),\\
        \partial_t v^\varepsilon &= B v^\varepsilon + g(u^\varepsilon, v^\varepsilon),
    \end{split}
\end{align}
where $\varepsilon > 0$ is a small parameter, $A$ and $B$ are linear and densely defined operators on Banach spaces $X$ and $Y$ of scalar-valued functions respectively. Moreover, $f$ and $g$ are sufficiently regular nonlinear functions and $(u^\varepsilon(t),v^\varepsilon(t))=(u^\varepsilon,v^\varepsilon)\in X\times Y$ are the unknown functions for some $t\in[0,\infty)$. We see that for $\varepsilon>0$ the fast-reaction system~\eqref{fast-reaction} is equivalent to a fast-slow system of the form
\begin{align}\label{fast-slow}
    \begin{split}
       \varepsilon \partial_t u^\varepsilon &=\varepsilon A u^\varepsilon +f(u^\varepsilon, v^\varepsilon),\\
        \partial_t v^\varepsilon &= B v^\varepsilon + g(u^\varepsilon, v^\varepsilon).
    \end{split}
\end{align}
The parameter $\varepsilon$ gives rise to a formal time-scale separation between the fast variable $u^\varepsilon$ and the slow variable $v^\varepsilon$. Letting $\varepsilon \to 0$ in either the fast-reaction equations \eqref{fast-reaction} or the fast-slow system \eqref{fast-slow}
 yields the following slow subsystem
 \begin{align}\label{slow system}
    \begin{split}
        0 &=f(u^0, v^0),\\
        \partial_t v^0 &= B v^0 + g(u^0, v^0),
    \end{split}
\end{align}
 which is an differential-algebraic equation defined on a critical set
 \begin{align} 
 \label{critical manifold}
     S_0:=\{(u^0,v^0)\in X\times Y\,:\, f(u^0,v^0)=0\},
 \end{align}
We refer to $S_0$ as the critical manifold, where we shall later on make a precise technical assumption involving the implicit function theorem to ensure that it is indeed a Banach manifold. \\

The main result of our work is a generalization of Fenichel's theory for the fast-reaction system \eqref{fast-reaction}.
The two key assumptions are that the critical manifold is attracting, which is encoded in suitable technical attraction conditions and that the operator $B$ allows a splitting of the space $Y$ into a fast and slow component.
This splitting is denoted by a small parameter $\zeta >0$, depending on $\varepsilon$.
Under these assumptions we obtain several results that together give a generalization of Fenichel's theory to infinite dimensions for the case of fast-reactions. The first result is that there exists a locally invariant and continuously differentiable manifold for the fast-reaction system \eqref{fast-reaction} denoted by $S_\varepsilon$.
Secondly, this slow manifold $S_\varepsilon$ converges to the critical manifold $S_0$ as $\varepsilon\to 0$.
Lastly, we show that the semi-flow on the slow manifold converges to the slow flow on a subset of the critical manifold $S_0$. The precise technical statements of these results can be found below in Propositions \ref{Prop 4.2}-\ref{Prop 4.9}. 
One key observation in the study of slow manifolds for fast-reaction systems of form \eqref{fast-reaction} is that the linearization of the non-linearity $f$ with respect to the fast variable $u$ determines the behaviour of the fast component. To encompass this in our system we rewrite it as follows
\begin{align}
    \begin{split}
         \partial_t u^\varepsilon &=\frac{1}{\varepsilon}\big(\varepsilon A +\textnormal{D}_u f(u,v)\big)  u^\varepsilon +\frac{1}{\varepsilon} \tilde f(u^\varepsilon, v^\varepsilon),\\
        \partial_t v^\varepsilon &= B v^\varepsilon + g(u^\varepsilon, v^\varepsilon),
    \end{split}
\end{align}
where the details are presented in Section 2.1 For this system we try to apply the results presented in \cite{hummel2022slow}. However, many results cannot be applied one-to-one as one has to be careful working with the new linear operator $\tilde A_\varepsilon=\varepsilon A+\textnormal{D}_u f(u^0,v^0) $, where $\textnormal{D}_u f$ denotes the Fr\'echet derivative in the direction of the first variable evaluated at $(u^0,v^0)$.

Now, let us briefly outline the structure of this work.
\begin{itemize}
    % \item We motivate the generalization of Fenichel's theorem by studying some ODE systems and highlight two different approaches for an infinite-dimensional extension (see Section \ref{section 2}).
    \item We show that the fast-reaction systems are well-posed and prove the convergence of solutions as the parameter         $\varepsilon\to 0$ (see Section \ref{section 3}).
    \item We generalize Fenichel's theorem for fast-reaction systems in infinite-dimensional Banach spaces.
          Under suitable assumptions we prove the existence of a $C^1$-regular slow manifold that is close to the critical manifold in a suitable sense.
          Moreover, we show that the critical manifold attracts solutions and that the flows on the slow manifold converge to the flow on the critical manifold.
          To conclude this part we apply the theory to an example (see Section \ref{section 4}).
    \item Lastly we present a brief overview on the theory of $C_0$ semigroups and their application to abstract Cauchy problems (see Appendix \ref{appendix A}). Moreover, we briefly present an alternative approach that uses a splitting type argument in both variables in order to control the $\varepsilon^{-1}-$terms (see Appendix \ref{Appendix B}).
\end{itemize}

\section{Fast Reaction Systems}\label{section 3}
In this section we analyse the general fast reaction system
\begin{align}\label{fast-slow 2}
    \begin{split}
        \partial_t u^\varepsilon &= A u^\varepsilon +\frac{1}{\varepsilon} f(u^\varepsilon, v^\varepsilon),\\
        \partial_t v^\varepsilon &= B v^\varepsilon + g(u^\varepsilon, v^\varepsilon),
    \end{split}
\end{align}
for $\varepsilon > 0$ and show the convergence of the semi-flow of the solution of the $\varepsilon$-system to the semi-flow of the solution of the limit system as $\varepsilon\to 0$.

\subsection{Function spaces and assumptions}\label{sect 3.1}

We structure the assumptions into several parts. First of, we have the problem setup, where 
\begin{itemize}
    \item  Let $0< \varepsilon\ll 1$ be a parameter that encompasses the separation of time scales, where in the limit as $\varepsilon\to 0$ the PDE system  \eqref{fast-slow 2} turns into an algebraic-differential equation.
    % \item Let $T>0$ be the final time of existence of solutions which is determined by the nonlinearity and the implicit function theorem.
    \item $X,Y$ are Banach spaces and we denote by $X_1= D(A)$ the domain of the operator $A$ and $Y_1= D(B)$ the domain of $B$ respectively.
    \item We set the initial data $(u_0,v_0)\in X_{1} \times Y_{1}$. Furthermore, we require that the initial values satisfy $f(u_0,v_0)=0$.
\end{itemize}
For the operators $A$ and $B$ we make the following assumptions
\begin{itemize}
    \item $A, B$ are closed linear operators, where $A:X\supset D(A)\to X$ generates the $C_0$-semigroup $(\textnormal{e}^{tA})_{t\geq 0}\subset \mathcal{B}(X)$ and $B:Y\supset D(B)\to Y$ generates the $C_0$-semigroup $(\textnormal{e}^{tB})_{t\geq 0}\subset \mathcal{B}(Y)$, respectively, where $\mathcal{B}(\cdot)$ denotes the space of bounded linear operators.
    \item The interpolation-extrapolation scales generated by $(X,A)$ and $(Y,B)$ are given by $(X_\alpha)_{\alpha\in [-1,\infty)}$ and $(Y_\alpha)_{\alpha\in [-1,\infty)}$ (for details regarding the definitions and basic properties of interpolation-extrapolation scales see Appendix \ref{appendix A} or \cite{amann1995linear}).
    \item Let $\gamma\in (0,1]$ if $(\textnormal{e}^{tA})_{t\geq 0}$ is analytic and $\gamma=1$ otherwise.
    Similarly, we introduce $\delta \in (0,1]$ if $(\textnormal{e}^{tB})_{t\geq 0}$ is analytic and set $\delta=1$ otherwise.
     \item There are constants $C_A,C_B, M_A, M_B>0$, $\omega_A \in \mathbb{R}$ and $\omega_B\in \mathbb{R}$ such that
    \begin{align*}
        \|\textnormal{e}^{tA}\|_{\mathcal{B}(X_{1})}&\leq M_A \textnormal{e}^{\omega_A t},\qquad \|\textnormal{e}^{tA}\|_{\mathcal{B}(X_\gamma,X_{1})}\leq C_At^{\gamma-1}\textnormal{e}^{\omega_A t},\\
        \|\textnormal{e}^{tB}\|_{\mathcal{B}(Y_{1})}&\leq M_B \textnormal{e}^{\omega_B t},\qquad \|\textnormal{e}^{tB}\|_{\mathcal{B}(Y_\delta,Y_{1})}\leq C_B t^{\delta-1} \textnormal{e}^{\omega_B t}
    \end{align*}
    hold for all $t\in (0,\infty)$.
\end{itemize}
\begin{remark}
    We observe that we do not need the semigroups to be exponentially stable, i.e. we do not need restrict $\omega_A,\omega_B$ to be negative. An example of an analytic semigroup with positive growth bound is the shifted Laplacian acting on $\Delta+ \textnormal{Id}$ on $H^2(\mathbb{T}^d)$. In this case we have $\omega= 1$.
\end{remark}
Next, we consider the assumptions for the nonlinear functions $f$ and $g$
\begin{itemize}
    \item Let $\alpha \in [0,1)$ be such that $0\leq \gamma- (1-\alpha)\leq 1$.
    The nonlinearity $f:X_{\alpha}\times Y_{1}\to X_{\gamma- (1-\alpha)}$ is a two-times continuously Fr\'echet differentiable function and $g:X_{1}\times Y_{1}\to Y_\delta$ is Fr\'echet differentiable. Moreover, there are constants $L_f, L_g>0$ such that
    \begin{align}
       \|\textnormal{D} f(x,y)\|_{\mathcal{B}(X_{1}\times Y_{1},X_\gamma)}&\leq L_f,\qquad  \|\textnormal{D} f(\tilde x,\tilde y)\|_{\mathcal{B}(X_{\alpha}\times Y_{\alpha},X_{\gamma-(1-\alpha)})}\leq L_f,\\
       \|\textnormal{D} g(x,y)\|_{\mathcal{B}(X_{1}\times Y_{1},Y_\delta)}&\leq L_g, 
    \end{align}
    for $(x,y)\in X_1\times Y_1$ and $(\tilde x,\tilde y)\in X_\alpha\times Y_\alpha$.
    \item This implies that the following inequalities hold
        \begin{align*}
        \|f(x_1,y_1)-f(x_2,y_2)\|_{X_\gamma}&\leq L_f( \|x_1-x_2\|_{X_{1}}+\|y_1-y_2\|_{Y_{1}}),\\
        \|f(\tilde x_1,\tilde y_1)-f(\tilde x_2,\tilde y_2)\|_{X_{\gamma-(1-\alpha)}}&\leq L_f( \|\tilde x_1-\tilde x_2\|_{X_{\alpha}}+\|\tilde y_1-\tilde y_2\|_{Y_{\alpha}}),\\
        \|g(x_1,y_1)-g(x_2,y_2)\|_{Y_\delta}&\leq L_g( \|x_1-x_2\|_{X_{1}}+\|y_1-y_2\|_{Y_{1}})
    \end{align*}
    for all $x_1,x_2\in X_{1}$, $y_1,y_2\in Y_{1}$. Here we assume that $f(x,y)\in X_\gamma$ if $(x,y)\in X_1\times Y_1 $. 
    \item For convenience, we assume that $f(0,0)=0$ and $g(0,0)=0$.
\end{itemize}

The next set of assumptions deals with the specific restrictions on the fast-reaction nonlinearity $f$.
\begin{itemize}
    \item Let $\mathcal{C}\subset X\times Y$ be a non-empty set such that for all $(x,y)\in \mathcal{C}$ we have $f(x,y)= 0$ and in addition the operator $\textnormal{D}_x f(x,y)\equiv \textnormal{D}  f(\cdot,y)x: X\to X_\gamma$ is a Banach space isomorphism, i.e. a bounded invertible linear operator.
   \end{itemize}
\begin{remark}
Then for each $(x_0,y_0) \in\mathcal{C}$ the implicit function theorem holds. 
That is, there exist open neighborhoods $V_{y_0}$ of $y_0$ and $U_{x_0}$ of $x_0$ and a Fr\'echet differentiable function $h^{y_0} :V_{y_0}\to U_{x_0}$ such that $f(h^{y_0}(y_0),y_0)=0$. Then $h^{y_0}$ satisfies $h^{y_0}(y)= x$ and $f(h^{y_0}(y),y)=0$ for all $(x,y)\in U_{x_0}\times V_{y_0}$.
Moreover $h^{y_0}$ has the same regularity as $f$.
\end{remark}
\begin{itemize} 
    \item As we are interested in an attractive slow manifold we now only consider all the $(x,y)\in \mathcal{C}$ such that the operator $\textnormal{D}_x f$ has its spectrum in the left half plane with an upper bound $\lambda_0<0$ on the real part and we denote this set by $\mathcal{C}_{-}$.
    \item If there exist $(x_1,y_1),(x_2,y_2)\in \mathcal{C}_{-}$ such that the neighborhoods introduced in the implicit function theorem satisfy $U_{x_1}\cap U_{x_2}\neq \emptyset$ and $V_{y_1}\cap V_{y_2}\neq \emptyset$.
    Then, there exists $(x_3,y_3)\in \mathcal{C}_{-}$ such that $U_{x_3}\subset U_{x_1}\cap U_{x_2} $ and $V_{y_3}\subset V_{y_1}\cap V_{y_2} $.
    Moreover, we have that $h^{y_1}(y)=h^{y_2}(y)$ for all $y\in V_{y_1}\cap V_{y_2}$.    
  \item Next, we assume that $\mathcal{C}_{-}$ can covered by a finite number of neighborhoods $(U_n,V_n)$ and we write $h^n:V_n\to U_n$, where $h^n(y)=x$ for all $(x,y)\in U_n\times V_n$ for the representative function of the implicit function theorem of each connected neighborhood.
\end{itemize}

\begin{remark}
   In many cases in the literature the fast-reaction term $f$ satisfies the above assumptions with $\mathcal{C}_{-}$ being covered with one neighborhood. See Section \ref{subsection example} for an example of this behaviour.
  However, it can evidently happen, e.g., consider the nonlinearity $f(x,y)=1-x^2-y^2$, where we need two charts to cover $\mathcal{C}_{-}$.
  As a third example we consider $f(x,y)= y\sin(x)$. Here, we need to restrict the variables to a compact set in order for the finite covering assumption to hold.
  Moreover, we observe that in this setting $\mathcal{C}_{-}$ consists of a finite number of disjoint subsets.
\end{remark}
In order to keep the computations as clear as possible and to avoid dealing with chart transforms, we restrict the nonlinearity $f$ to the case of one neighborhood $(U_0,V_0)$ covering $\mathcal{C}_{-}$.

Next, we use the above assumptions on the nonlinearity $f$ to write it as a Taylor approximation close to a point $(u^0,v^0)\in U_0\times V_0\subset \mathcal{C}_{-}$ and we obtain
\begin{align}\label{eq: Taylor approx}
    f(u,v)=\underbrace{f(u^0,v^0)}_{=0}+ \textnormal{D}_x f(u^0,v^0)(u-u^0)+ \textnormal{D}_y f(u^0,v^0)(v-v^0)+ \mathcal{O}((u-u^0,v-v^0)^2).
\end{align}
For a precise definition of the Taylor approximation in Banach spaces we refer to \cite{zeidler2012applied}.
Then, we can introduce the function $\tilde f$ as $\tilde f(u,v):=f(u,v)-\textnormal{D}_x f(u^0,v^0)u$ and we can rewrite the fast-reaction system in a modified form
\begin{align}\label{modified fast reaction}\begin{split}
    \partial_t u^\varepsilon&= \frac{1}{\varepsilon}\tilde A_\varepsilon u^\varepsilon +\frac{1}{\varepsilon}\tilde f(u^\varepsilon,v^\varepsilon),\\
    \partial_t v^\varepsilon&= B v^\varepsilon+ g(u^\varepsilon,v^\varepsilon),
\end{split}
\end{align}
where $\tilde A_\varepsilon:= \varepsilon A+ \textnormal{D}_x f(u^0,v^0)$.

With the above assumptions and a perturbation result for semigroups (see Appendix \ref{appendix A}) we obtain that the operator $\tilde A_\varepsilon$ satisfies the following estimates
\begin{align*}
     \|\textnormal{e}^{t \tilde A_\varepsilon}\|_{\mathcal{B}(X_{1})}&\leq M_{\tilde A} \textnormal{e}^{(\varepsilon\omega_A+\lambda_0) t},\qquad \|\textnormal{e}^{t \tilde A_\varepsilon}\|_{\mathcal{B}(X_\gamma,X_{1})}\leq C_{\tilde A}t^{\gamma-1}\textnormal{e}^{(\varepsilon\omega_A+\lambda_0) t}
\end{align*}
for some constants $M_{\tilde A},\, C_{\tilde A}>0$ and all $t>0$.
Moreover, we obtain that $\tilde f$ has the same regularity as $f$ and the Lipschitz constant satisfies 
$$ L_{\tilde f} \leq L_f +\sup_{0\leq t<\infty}\|\textnormal{D}_x f(u^0(t),v^0(t))\|_{\mathcal{B}(X_{1}\times Y_1, X_\gamma)}.$$
For this new system we require to the following to hold
\begin{itemize}
    \item The growth bound for the semigroup generated by the operator $\tilde A_\varepsilon$ satisfies $\varepsilon\omega_A+\lambda_0<0$, i.e. the operator $\tilde A_\varepsilon$ generates an exponentially stable semigroup.
    \item Lastly, we assume that 
    $$L_{\tilde f}\inf_{0\leq t< \infty}\|\big(\textnormal{D}_x f(u^0(t),v^0(t))\big)^{-1}\|_{\mathcal{B}(X_{1}\times Y_1, X_\gamma)}<1.$$
    This guarantees that the function $h^0$ is indeed Lipschitz continuous with constant $L_{h^0}$.
\end{itemize}

\subsection{Existence of solutions}
In this section we show the well-posedness of a general fast reaction system for $\varepsilon>0$ 
\begin{align}\begin{split}\label{full eq}
    \partial_t u^\varepsilon(t) &=A u^\varepsilon(t) + \frac{1}{\varepsilon} f( u^\varepsilon(t),v^\varepsilon(t)),\\
    \partial_t v^\varepsilon (t) &
    =  B v^\varepsilon(t) +g(u^\varepsilon(t),v^\varepsilon(t)),\\
    u^\varepsilon(0)&=  u_0,\quad v^\varepsilon(0)=v_0,
\end{split}
\end{align}
and its limit system when $\varepsilon=0$
\begin{align}\begin{split}\label{full eq limit}
  0&=  f(u^0(t),v^0(t)),\\
    \partial_t v^0 (t) &
    =  B v^0(t) +g(u^0(t),v^0(t)),\\
    u^0(0)&= u_0,\qquad v^0(0)=v_0.
\end{split}
\end{align}
This system is the standard form for fast reaction systems found in the literature. And in the subsequent sections all results will be posed for systems \eqref{full eq} and \eqref{full eq limit}.
The modifies system \eqref{modified fast reaction} is introduced to have a better control on the $\varepsilon$-dependent terms in the proofs. 

\begin{remark}
    For simplicity we assume that for equations \eqref{full eq} and \eqref{full eq limit} we have the same initial data $(u_0,v_0)\in X_1\times Y_1$. However, it is also possible to consider different sets of initial data for the two systems. In this case we would require that $(u_0^\varepsilon,v_0^\varepsilon)\to (u_0^0,v_0^0)$ in $X_1\times Y_1$ as $\varepsilon\to 0$.
\end{remark}

\begin{prop}\label{Prop 3.1}
Let $\varepsilon>0$ be small. Then, there exists a unique strict solution $(u^\varepsilon,v^\varepsilon)$ of the system  \eqref{full eq}, i.e the solution satisfies $(u^\varepsilon,v^\varepsilon)\in C^1([0,\infty);X\times Y) \cap C([0,\infty);X_1\times Y_1)$.
\end{prop}

\begin{proof}
The idea of the proof is to construct a contracting solution operator and apply Banach's fixed point theorem on a suitable space. 
The space we consider is $C_b([0,\infty), \textnormal{e}^{\eta t};X_1\times Y_1)$, i.e. the space of all $(u,v)\in C([0,\infty); X_1\times Y_1)$ such that
\begin{align*}
    \|(u,v)\|_{C_b([0,\infty), \textnormal{e}^{\eta t};X_1\times Y_1)}:= \sup_{t\geq 0} \textnormal{e}^{-\eta t} \big(\|u(t)\|_{X^1}+\|v(t)\|_{Y_1}\big)<\infty,
\end{align*}
where $\eta \geq 0$ and the precise value will be determined later.
On this space we now aim to show that the operator
\begin{align*}
    \mathcal{L}(u,v):= \begin{pmatrix}  
        \textnormal{e}^{  t A} u_0 & \quad +\varepsilon^{-1} \int_0^t \textnormal{e}^{ (t-s) A}  f(u(s),v(s)) \, \textnormal{d} s\\
     \textnormal{e}^{tB}v_0 &+ \int_0^t \textnormal{e}^{(t-s) B} g(u(s),v(s))\, \textnormal{d} s.
    \end{pmatrix}
\end{align*}
is a contraction.
Thus, estimating $\|\mathcal{L}(u,v)-\mathcal{L}(\tilde u,\tilde v)\|_{X_1\times Y_1}$ yields
\begin{align*}
     &\sup_{t\geq 0} \textnormal{e}^{-\eta t}\bigg\|\varepsilon^{-1} \int_{0}^t \textnormal{e}^{(t-s)A}  [f(u(s), v(s))- f(\tilde u(s), \tilde v(s))]\, \textnormal{d} s\bigg\|_{X_1}\leq \\
     \leq & L_fC_A \varepsilon^{-1}\int_{0}^t \frac{\textnormal{e}^{(t-s)(\omega_A-\eta)}}{(t-s)^{1-\gamma_X}}\, \textnormal{d}s \|(u(s), v(s))-(\tilde u(s), \tilde v(s))\|_{C_b([0,\infty), \textnormal{e}^{\eta t};X_1\times Y_1)}\\
     \leq &  L_fC_A \varepsilon^{-1}\frac{\Gamma(\gamma_X)}{(\eta -\omega_A)^{\gamma_X}} \|(u(s), v(s))-(\tilde u(s), \tilde v(s))\|_{C_b([0,\infty), \textnormal{e}^{\eta t};X_1\times Y_1)}
\end{align*}
and 
\begin{align*}
     &\sup_{t\geq 0} \textnormal{e}^{-\eta t}\bigg\| \int_{0}^t \textnormal{e}^{(t-s)B}  [g(u(s), v(s))- g(\tilde u(s), \tilde v(s))]\, \textnormal{d} s\bigg\|_{Y_1}\leq \\
     \leq & L_g C_B \int_{0}^t \frac{\textnormal{e}^{(t-s)(\omega_B-\eta)}}{(t-s)^{1-\gamma_Y}}\, \textnormal{d}s \|(u(s), v(s))-(\tilde u(s), \tilde v(s))\|_{C_b([0,\infty), \textnormal{e}^{\eta t};X_1\times Y_1)}\\
     \leq &  L_g C_B \frac{\Gamma(\gamma_Y)}{(\eta -\omega_B)^{\gamma_Y}} \|(u(s), v(s))-(\tilde u(s), \tilde v(s))\|_{C_b([0,\infty), \textnormal{e}^{\eta t};X_1\times Y_1)}.
\end{align*}
For $\eta$ large enough we have that $\mathcal{L}$ is a contraction and thus there exists a unique fixed point 
\[(u^\varepsilon_*, v^\varepsilon_*)\in C_b([0,\infty), \textnormal{e}^{\eta t}; X_1\times Y_1).\]
Let $(u^\varepsilon,v^\varepsilon)$ be this fixed point. 
Then, we have that
\begin{align*}
    u^\varepsilon(t)&= \textnormal{e}^{tA}u_0+ \frac{1}{\varepsilon}\int_0^t \textnormal{e}^{(t-s)A} f(u^\varepsilon(s),v^\varepsilon(s))\, \textnormal{d}s,\\
    v^\varepsilon(t)&= \textnormal{e}^{tB}v_0+ \int_0^t \textnormal{e}^{(t-s)B} g(u^\varepsilon(s),v^\varepsilon(s))\, \textnormal{d}s
\end{align*}
which implies that
\begin{align*}
    u^\varepsilon(t) &= u_0 + A\int_0^t u^\varepsilon(s)\, \textnormal{d}s +\frac{1}{\varepsilon} \int_0^t f( u^\varepsilon(s),v^\varepsilon(s))\, \textnormal{d}s,\\
    v^\varepsilon(t) &= v_0 + B\int_0^t v^\varepsilon(s)\, \textnormal{d}s + \int_0^t g( u^\varepsilon(s),v^\varepsilon(s))\, \textnormal{d}s
\end{align*}
see \cite{lunardi2012analytic} for details.
Thus, for all $t\in [0,\infty)$ we conclude that
\begin{align*}
    \partial_t u^\varepsilon(t)=\lim_{h\to 0} \frac{u^\varepsilon(t+h) -u^\varepsilon(t)}{h}&=\lim_{h\to 0}\bigg[ \int_t^{t+h} A u^\varepsilon(s)\, \textnormal{d}s +\frac{1}{\varepsilon} \int_t^{t+h} f( u^\varepsilon(s),v^\varepsilon(s))\, \textnormal{d}s \bigg]\\
\intertext{is well-defined and in the limit}
 \partial_t u^\varepsilon(t)&=A u^\varepsilon(t) + \frac{1}{\varepsilon} f(u^\varepsilon(t),v^\varepsilon(t)),
\end{align*}
where the convergence holds in $X$ as $A u^\varepsilon,\, f(u^\varepsilon,v^\varepsilon)\in C([0,\infty);X)$.
With the same idea we can show the well-posedness of $\partial_t v^\varepsilon\in C([0,\infty);Y)$.
Hence, we have shown that
\begin{align*}
    (u^\varepsilon, v^\varepsilon)\in C([0,\infty); X_1\times Y_1)\cap C^1([0,\infty); X\times Y)
\end{align*}
is the unique solution to \eqref{full eq} for $\varepsilon>0$.
\end{proof} 

\begin{prop}\label{Prop 3.2}
    Let $\varepsilon=0$. Then under the assumptions of Section \ref{sect 3.1} there exists a unique strict solution $(u^0,v^0)\in C^1([0,\infty);X\times Y)\cap C([0,\infty);X_1\times Y_1)$ to equation \eqref{full eq limit}.
\end{prop}
  
\begin{proof}
From the assumptions on the initial data and the nonlinearity $f$ we have that $f(u_0,v_0)=0$.
Then, following the observations from the previous section we can apply the implicit function theorem at $(u_0,v_0)\in X_1\times Y_1$ and obtain the existence of a neighborhood $(U_0,V_0)\subset \mathcal{C}_{-}$ and a Fr\'{e}chet differentiable function $h^0:V_{0}\subset Y_1\to  U_0 \subset X_1$ such that $h^0(y)=x$ for all $(x,y)\in U_0\times V_0$.
Moreover, as $\textnormal{D}_x f(u_0,v_0) h^0(y)+ \tilde f(h^0(y),y)=0$ we can write $h^0(y)= -\textnormal{D}_x f(u_0,v_0)^{-1}\tilde f(h^0(y),y) $.
Next, we proof that $h^0$ is indeed Lipschitz continuous. To see this, let $y_1,y_2\in V_0$ and we compute
\begin{align*}
    \|h^0(y_1)-h^0(y_2)\|_{X_1}\leq L_{\tilde f}\|(\textnormal{D}_x f(x^*,y^*))^{-1}\|_{\mathcal{B}(X_\gamma,X_1)}( \|h^0(y_1)-h^0(y_2)\|_{X_1}+\|y_1-y_2\|_{Y_1})
\end{align*}
which yields
\begin{align*}
    \|h^0(y_1)-h^0(y_2)\|_{X_1}\leq \underbrace{L_{\tilde f}\|(\textnormal{D}_x f(x^*,y^*))^{-1}\|_{\mathcal{B}(X_\gamma,X_1)} \big(1-L_{\tilde f}\|(\textnormal{D}_x f(x^*,y^*))^{-1}\|_{\mathcal{B}(X_\gamma,X_1)}\big)^{-1} }_{=:L_{h^0}}\|y_1-y_2\|_{Y_1}.
\end{align*}

% Using the assumptions on the nonlinearity $f$ the implicit function theorem applied to the level set $f(u^0(t),v^0(t))=0$ implies that there exists a Fr\'{e}chet differentiable function $h^0:U_{Y_1}\subset Y_1\to X_1$ such that $h^0(y)=x$ and $h^0(y)= -\textnormal{D}_x f(x^*,y^*)^{-1}\tilde f(h^0(y),y) $ for each $t\in [0,\infty)$.
% \textcolor{blue}{To see that $h^0$ is Lipschitz continuous we compute}
% \begin{align*}
%     \|h^0(y_1)-h^0(y_2)\|_{X_1}\leq L_{\tilde f}\|(\textnormal{D}_x f(x^*,y^*))^{-1}\|_{\mathcal{B}(X_\gamma,X_1)}( \|h^0(y_1)-h^0(y_2)\|_{X_1}+\|y_1-y_2\|_{Y_1})
% \end{align*}
% which yields
% \begin{align*}
%     \|h^0(y_1)-h^0(y_2)\|_{X_1}\leq \underbrace{L_{\tilde f} \big(1-L_{\tilde f}\|(\textnormal{D}_x f(x^*,y^*))^{-1}\|_{\mathcal{B}(X_\gamma,X_1)}\big)^{-1} }_{=:L_{h^0}}\|y_1-y_2\|_{Y_1}.
% \end{align*}
Thus, the mapping $g: Y_1\to Y_{\delta},\quad y \mapsto g(h^0(y),y)$ is Lipschitz continuous and satisfies 
\begin{align*}
\|g(h^0(y_1),y_1)-g(h^0(y_2),y_2)\|_{Y_\delta} &\leq L_g\|h^0(y_1)-h^0(y_2)\|_{Y_{1}}+ L_g\|y_1-y_2\|_{Y_{1} }\\
&\leq L_g(1+L_{h^0})\|y_1-y_2\|_{Y_{1}}.
\end{align*}
Then, there exists a unique strict solution $v^0\in C^1([0,\infty);Y)\cap C([0,\infty);Y_{1}) $ to 
\begin{align*}
\partial_t v^0= B v^0 + g(h^0(v^0),v^0),\qquad v^0(0)=v_0.
\end{align*}
Indeed, using variation by constants and a suitable weighted space, similar to the one in Proposition \ref{Prop 3.1}, yields that the operator 
\begin{align*}
    \mathcal{L}_{B}(v^0)= \textnormal{e}^{Bt} v_0+ \int_0^t \textnormal{e}^{B(t-s)} g(h^0(v^0),v^0)\, \textnormal{d}s
\end{align*}
has a fixed point.
Now, setting $u^0(t):= h^0(v^0(t))$ we obtain that $u^0\in C^1([0,\infty);X)\cap C([0,\infty);X_{1})$.
Thus, $(u^0,v^0)\in C^1([0,\infty);X\times Y)\cap C([0,\infty);X_{1}\times Y_{1})$ is the unique strict solution of the limit system \eqref{full eq limit} with $\varepsilon=0$.
\end{proof}

\begin{remark}
    Now, we can make the definition of the modified system \eqref{modified fast reaction} precise.
    The idea is to use the solution $(u^0,v^0)$ obtained in Proposition \ref{Prop 3.2} as the starting point for the Taylor approximation of the nonlinearity $f$ in \eqref{eq: Taylor approx}.
    Then, the existence and uniqueness of solutions to \eqref{modified fast reaction} follows from Proposition \ref{Prop 3.1}.
\end{remark}

We conclude this section with two observations.
\begin{remark}
The mapping 
\[h^0: Y_{1}\to X_{1},\qquad y\mapsto h^0(x),\] 
satisfying $f(h^0(y),y)=0$, introduced in the previous proof describes the critical manifold $S_0$ over $Y_{\alpha_2}$ given by
\[S_0:=\{ (h^0(y),y):\, y\in Y_{\alpha_2}\}\subset X\times Y.\]

\end{remark}
\begin{remark}
  We note that both systems \eqref{full eq} and \eqref{full eq limit} are autonomous.
 Therefore, the solutions $(u^0,v^0)$ for $\varepsilon=0$ and $(u^\varepsilon,v^\varepsilon)$ for $\varepsilon>0$ are given by semiflows, i.e. continuous mappings of the form
 \[T_0:[0,\infty)\times S_0\to S_0,\qquad T_\varepsilon:[0,\infty)\times X_1\times Y_1\to X_1\times Y_1\]
and we write
\begin{align*}
    \begin{pmatrix}
    u^0(t)\\v^0(t)
    \end{pmatrix} =T_0(t)\begin{pmatrix}
    h(v_0)\\v_0
    \end{pmatrix},\qquad \begin{pmatrix}
    u^\varepsilon(t)\\v^\varepsilon(t)
    \end{pmatrix}=T_\varepsilon (t)\begin{pmatrix}
    u_0\\v_0
    \end{pmatrix}.
\end{align*}
\end{remark}

\subsection{Approximation by slow flow}

The main result of this subsection is to show the convergence of the solutions of the full system \eqref{modified fast reaction} to the solutions of the limit system \eqref{full eq limit} as $\varepsilon\to 0$.

\begin{theorem}\label{Thm 3.6}
There are constants $C_1,\,C_2>0$ such that
\begin{align*}
    \bigg\| T_\varepsilon(t)\begin{pmatrix}
    u_0\\v_0
    \end{pmatrix}- T_0(t)\begin{pmatrix}
    h^0(v_0)\\v_0
    \end{pmatrix}\bigg \|_{X_\alpha\times Y_1}&\leq C_1\textnormal{e}^{\varepsilon^{-1}(\varepsilon\omega_A+\lambda_0) t}\|u_0- h^0(v_0)\|_{X_\alpha}\\
    &\quad+  C_2 \varepsilon^{1-\alpha}  (\textnormal{e}^{\omega_B t}  + \varepsilon^{-\gamma}t^\gamma\textnormal{e}^{\varepsilon^{-1}(\varepsilon\omega_A+\lambda_0) t}) \|v_0\|_{Y_1}
\end{align*}
holds for all $(u_0,v_0)\in X_{\alpha}\times Y_{1}$, where $0\leq\alpha<1$, $t\in [0,\infty)$ and all $\varepsilon>0$.
\end{theorem}
\begin{proof}
In the following we use the classical notion of the solution rather than the semi-flow version, i.e. we have
\begin{align*}
    \begin{pmatrix}
    u^\varepsilon(t)\\v^\varepsilon(t)
    \end{pmatrix}=T_\varepsilon(t)\begin{pmatrix}
    u_0\\v_0
    \end{pmatrix},\quad \begin{pmatrix}
    h^0(v^0)(t)\\v^0(t)
    \end{pmatrix}=T_{0}(t)\begin{pmatrix}
    h^0(v_0)\\v_0
    \end{pmatrix}.
\end{align*}
Moreover, to control the $\varepsilon$-dependent terms we will work with the modified system \eqref{modified fast reaction}.
Then, using the variation of constants formula the solutions to systems \eqref{modified fast reaction} and \eqref{full eq limit} can be written as
% From the existence of solutions we obtain that the following equations are well-posed
\begin{align*}
    u^\varepsilon(t) &= \textnormal{e}^{\varepsilon^{-1} t\tilde A_\varepsilon} u_0+ \varepsilon^{-1} \int_{0}^t \textnormal{e}^{\varepsilon^{-1}(t-s)\tilde A_\varepsilon} \tilde f(u^\varepsilon(s), v^\varepsilon(s)) \, \textnormal{d} s,\\
    v^\varepsilon(t)&= \textnormal{e}^{tB}v_0+ \int_0^t \textnormal{e}^{(t-s)B} g(u^\varepsilon(s),v^\varepsilon(s))\, \textnormal{d}s,
       \intertext{and }
    0&= \textnormal{D}_x f(h^0(v^0),v^0) h^0(v^0) + \tilde f(h^0(v^0),v^0),\\
    v^0(t)&= \textnormal{e}^{tB}v_0+ \int_0^t \textnormal{e}^{(t-s)B} g(h(v^0(s)),v^0(s))\, \textnormal{d}s.
\end{align*}
By the existence results (Prop. \ref{Prop 3.1} and Prop. \ref{Prop 3.2}) these equations are well-defined.
In addition, estimating the last equation yields
\begin{align*}
    \|v^0(t)\|_{Y_1}\leq M_B \textnormal{e}^{\omega_B t} \|v_0\|_{Y_1}+ C_B L_g \int_0^t \frac{\textnormal{e}^{(t-s)\omega_B}}{(t-s)^{1-\delta}} (1+L_h)\|v^0(s)\|_{Y_1}  \, \textnormal{d} s.
\end{align*}
This can be further estimated by applying Gronwall's inequality
\begin{align}\label{ineq v0}
     \|v^0(t)\|_{Y_1}\leq M_B \textnormal{e}^{\omega_B t} \textnormal{e}^{C_BL_g(1+L_h)\Gamma(\delta)} \|v_0\|_{Y_1}.
\end{align}
In a similar fashion we obtain that $\partial_t v^0$ satisfies
\begin{align*}
    \partial_t  v^0(t)= \textnormal{e}^{Bt}\big(Bv_0+ g(h^0(v_0),v_0)\big) +\int_0^t \textnormal{e}^{B(t-s)}\partial_s g(h^0(v^0(s)),v^0(s))\, \textnormal{d}s 
\end{align*}
and we estimate
\begin{align*}
    \|\partial_t  v^0\|_Y&\leq M_B \textnormal{e}^{\omega_B t} \big(\|Bv_0\|_Y+ L_g(L_h+1)\|v_0\|_Y\big) +M_B\int_0^t \textnormal{e}^{\omega_B(t-s)}\|\textnormal{D}g (h^0(v^0),v^0)\|_{\mathcal{B}(X_1\times Y_1,Y)} \|\partial_s v^0\|_Y \, \textnormal{d}s 
    \end{align*}
   which yields, assuming that the total derivative of $g$ is bounded by $L_g$,
   \begin{align}\label{ineq partial t v0}
    \|\partial_t  v^0\|_Y&\leq C M_B \textnormal{e}^{\omega_B t} \big(\|Bv_0\|_Y+ L_g(L_h+1)\|v_0\|_Y\big).
\end{align}

To estimate the fast component we introduce two approximate systems in $u$, namely $\tilde u^\varepsilon$ and $u^{\varepsilon,0}$, where $\tilde u^\varepsilon$ is the solution of
\begin{align*}
    \partial_t \tilde u^\varepsilon =\frac{1}{\varepsilon} \tilde A_\varepsilon u^\varepsilon +\frac{1}{\varepsilon}\tilde f(\tilde u^\varepsilon, v^0)
\end{align*}
and $u^{\varepsilon,0}$ is the solution of
\begin{align*}
    \partial_t u^{\varepsilon,0}= \frac{1}{\varepsilon} \tilde A_\varepsilon u^{\varepsilon,0}+\frac{1}{\varepsilon} \tilde f(u^{\varepsilon,0},v^0)+\partial_t \big(\tilde{A}_\varepsilon\big)^{-1}\tilde f(h^0(v^0),v^0)
\end{align*}
and where in both cases $v^0$ satisfies the second equation of the system \eqref{full eq limit}.
The expression $\big(\tilde{A}_\varepsilon\big)^{-1}$ is well defined since the resolvent of $A$ is positive and the spectrum of $\textnormal{D}_x f(h^0(v^0),v^0)$ is by assumption completely on the left half plane.
The existence of solutions follows from adapting the proofs of Propositions \ref{Prop 3.1} and \ref{Prop 3.2}.

We start with the estimates for the $u$ component in the norm $\|\cdot\|_{X_\alpha}$, where $\alpha \in (0,1)$ is chosen such that $0\leq \gamma-(1-\alpha)\leq 1$.
The reason for the choice of norm becomes clear when looking at the details of the estimate.
We apply the triangle inequality to split the estimate into three parts
\begin{align*}
    \|u^\varepsilon(t) -h^0(v^0(t))\|_{X_\alpha}\leq \underbrace{\|u^\varepsilon(t) -\tilde u^\varepsilon(t)\|_{X_\alpha} }_{=\mathcal{ I}}+ \underbrace{\|\tilde u^\varepsilon(t) -u^{\varepsilon,0}(t)\|_{X_\alpha}}_{=\mathcal{II}}+\underbrace{\|u^{\varepsilon,0}(t) -h^0(v^0(t))\|_{X_\alpha}}_{=\mathcal{III}}.
\end{align*}
For $\mathcal{I}$ we estimate  
\begin{align*}
    \|u^\varepsilon(t) -\tilde u^\varepsilon(t)\|_{X_\alpha}&\leq \bigg\| \varepsilon^{-1} \int_0^t \textnormal{e}^{\varepsilon^{-1}(t-s)\tilde A_\varepsilon}( \tilde f(u^\varepsilon,v^\varepsilon)-\tilde f(\tilde u^\varepsilon, v^0))\, \textnormal{d}s\bigg\|_{X_\alpha}\\
    &\leq L_{\tilde f} C_A \int_0^t \frac{\textnormal{e}^{(t-s)\varepsilon^{-1}(\varepsilon\omega_A+\lambda_0)}}{\varepsilon^\gamma (t-s)^{1-\gamma}}\big(\|u^\varepsilon(s)-\tilde u^{\varepsilon}(s)\|_{X_\alpha}+  \|v^\varepsilon(s)-v^0(s)\|_{Y_1}\big)\, \textnormal{d}s\\
    &\leq \frac{ L_{\tilde f} C_A  \Gamma(\gamma)}{|\varepsilon \omega_A-\lambda_0|^\gamma} \sup_{0\leq s\leq t}\|v^\varepsilon(s)-v^0(s)\|_{Y_1} + L_{\tilde f} C_A \int_0^t \frac{\textnormal{e}^{(t-s)\varepsilon^{-1}(\varepsilon\omega_A+\lambda_0)}}{\varepsilon^\gamma (t-s)^{1-\gamma}} \|u^\varepsilon(s)-\tilde u^{\varepsilon}(s)\|_{X_\alpha}  \, \textnormal{d}s
\end{align*}
  which by Gronwall's inequality yields
 \begin{align}\label{est u1}
   \|u^\varepsilon(t) -\tilde u^\varepsilon(t)\|_{X_\alpha}\leq C_AL_{\tilde f} C_1 \frac{\Gamma(\gamma)}{|\varepsilon \omega_A-\lambda_0|^\gamma} \sup_{0\leq s\leq t}\|v^\varepsilon(s)-v^0(s)\|_{Y_1},
\end{align}
   where $
   C_1=\textnormal{e}^{ C_AL_{\tilde f} \frac{\Gamma(\gamma)}{|\varepsilon \omega_A-\lambda_0|^\gamma}}$.\\
For $\mathcal{II}$ we have
\begin{align*}
     \|\tilde u^{\varepsilon}(t)-u^{\varepsilon,0}(t)\|_{X_\alpha}&\leq \bigg\| \varepsilon^{-1} \int_0^t \textnormal{e}^{\varepsilon^{-1}(t-s)\tilde A_\varepsilon}\big( \tilde f(\tilde u^\varepsilon,v^0)-\tilde f( u^{\varepsilon,0}, v^0) -\varepsilon \partial_s \big(\tilde A_\varepsilon\big)^{-1}\tilde f(h^0(v^0),v^0)\big) \, \textnormal{d}s \bigg\|_{X_\alpha}\\
    &\leq C_A L_{\tilde f}\int_0^t \frac{\textnormal{e}^{(t-s)\varepsilon^{-1}(\varepsilon \omega_A -\lambda_0)}}{\varepsilon^\gamma (t-s)^{1-\gamma} }\|u^\varepsilon(s)-u^{\varepsilon,0}(s)\|_{X_\alpha}\, \textnormal{d}s\\
    &\quad+ \varepsilon C_A \int_0^t  \frac{\textnormal{e}^{(t-s)\varepsilon^{-1}(\varepsilon \omega_A -\lambda_0)}}{\varepsilon^\gamma (t-s)^{1-\gamma} }\big\|\big(\tilde A_\varepsilon\big)^{-1} \partial_s  \tilde f(h^0(v^0),v^0)\big \|_{X_\gamma-(1-\alpha)}\, \textnormal{d}s\\
     &\quad + \varepsilon  C_A L_{\tilde f} \int_0^t  \frac{\textnormal{e}^{(t-s)\varepsilon^{-1}(\varepsilon \omega_A -\lambda_0)}}{\varepsilon^\gamma (t-s)^{1-\gamma} } \|\partial_s (\tilde A_\varepsilon)^{-1}\|_{\mathcal{B}(X_{\gamma-(1-\alpha)})}(1+ L_{h^0})\|v^0\|_{Y_1}\, \textnormal{d}s \\
    & \leq C_A L_{\tilde f}\int_0^t \frac{\textnormal{e}^{(t-s)\varepsilon^{-1}(\varepsilon \omega_A -\lambda_0)}}{\varepsilon^\gamma (t-s)^{1-\gamma} }\|u^\varepsilon(s)-u^{\varepsilon,0}(s)\|_{X_\alpha}\, \textnormal{d}s\\
    &\quad+  \frac{\varepsilon C_A L_{\tilde f} }{|\varepsilon\omega_A+\lambda_0|^\gamma} \|\big(\tilde A_\varepsilon\big)^{-1}\|_{\mathcal{B}(X,X_{\gamma-(1-\alpha)})} \|\textnormal{D}\tilde f (h^0(v^0),v^0)\|_{\mathcal{B}(X\times Y,X_{\gamma-(1-\alpha)})}\\
    &\qquad \times\big(\|\partial_t h^0(v^0) \|_{L^\infty([0,t],X)} +\|\partial_t v^0\|_{L^\infty([0,\infty),Y)}\big)\\
    &\quad + \frac{\varepsilon C_A L_{\tilde f}(1+ L_{h^0}) }{|\varepsilon\omega_A+\lambda_0|^\gamma} \|\big(\tilde A_\varepsilon\big)^{-1}\|_{\mathcal{B}(X,X_{\gamma-(1-\alpha)})} \|\big(\tilde A_\varepsilon\big)^{-1}\|_{\mathcal{B}(X)}
    \|\textnormal{D}^2 f (h^0(v^0),v^0)\|_{\mathcal{B}(X\times Y, X)}\\
    &\qquad \times\big(\|\partial_t h^0(v^0) \|_{L^\infty([0,t],X)} +\|\partial_t v^0\|_{L^\infty([0,t],Y)}\big)\|v^0\|_{Y_1}\\
    &\leq C_A L_{\tilde f}\int_0^t \frac{\textnormal{e}^{(t-s)\varepsilon^{-1}(\varepsilon \omega_A -\lambda_0)}}{\varepsilon^\gamma (t-s)^{1-\gamma} }\|u^\varepsilon(s)-u^{\varepsilon,0}(s)\|_{X_\alpha}\, \textnormal{d}s\\
    &\quad +\frac{\varepsilon C C_A L_{\tilde f} }{|\varepsilon\omega_A+\lambda_0|^\gamma} (\varepsilon^{-\gamma+(1-\alpha)}+ \textnormal{e}^{\omega_B t}\|v_0\|_{Y_1})\big(\|\partial_t h^0(v^0) \|_{L^\infty([0,t],X)} +\|\partial_t v^0\|_{L^\infty([0,t],X)}\big)
    \end{align*}
    where we used the estimate $\|\big(\tilde A_\varepsilon\big)^{-1}\|_{\mathcal{B}(X,X_{\gamma-(1-\alpha)})}\leq \varepsilon^{-\gamma+(1-\alpha)}$, see Remark \ref{remark 3.10} for more details.
    Then, by applying Gronwall's inequality we obtain
    \begin{align*}
    \|\tilde u^{\varepsilon}(t)-u^{\varepsilon,0}(t)\|_{X_\alpha}&\leq   \varepsilon  \frac{C C_A L_{\tilde f}\Gamma^(\gamma) }{|\varepsilon\omega_A+\lambda_0|^\gamma} (\varepsilon^{-\gamma+(1-\alpha)}+ \textnormal{e}^{\omega_B t}\|v_0\|_{Y_1}) \big(\|\partial_t h^0(v^0) \|_{L^\infty([0,t],X)} +\|\partial_t v^0\|_{L^\infty([0,t],Y)}\big).
    \end{align*}
To improve this estimate we note that
\begin{align*}
    \|\partial_t h^0(v^0)\|_{L^\infty([0,t],X)}&\leq  \|h^0( v^0)\|_{C^1_b([0,t],X)}\leq L_{h^0} \|v^0\|_{C^1_b([0,t],Y)}
    \intertext{and by applying inequalities \eqref{ineq v0} and \eqref{ineq partial t v0} we obtain}
     \|\partial_t h^0(v^0)\|_{L^\infty([0,t],X)}&\leq C \textnormal{e}^{\omega_B t}\big( \|v_0\|_{Y_1} +\|Bv_0\|_{Y}+ L_g(L_{h^0}+1)\|v_0\|_{Y}\big)  \leq C \textnormal{e}^{\omega_B t}  \|v_0\|_{Y_1}.
\end{align*}
Hence, we can conclude that
\begin{align}\label{est u2}
    \|\tilde u^{\varepsilon}(t)-u^{\varepsilon,0}(t)\|_{X_\alpha}&\leq   \varepsilon \frac{C C_A L_{\tilde f} \Gamma(\gamma)}{|\varepsilon\omega_A+\lambda_0|^\gamma}(\varepsilon^{-\gamma+(1-\alpha)}+ \textnormal{e}^{\omega_B t}\|v_0\|_{Y_1})  \textnormal{e}^{\omega_B t}  \|v_0\|_{Y_1}.
\end{align}
For $\mathcal{III}$ we have that
\begin{align*}
      \|u^{\varepsilon,0}(t)-h^0(v^0(t))\|_{X_\alpha}&= 
      \bigg\|\textnormal{e}^{t \varepsilon^{-1} \tilde A_\varepsilon} u_0+\varepsilon^{-1}\int_0^t \textnormal{e}^{(t-s)  \varepsilon^{-1}\tilde A_\varepsilon}  \tilde f(u^{\varepsilon,0}(s) ,v^0(s))\, \textnormal{d} s\\
      &\qquad +\int_0^t \textnormal{e}^{(t-s)  \varepsilon^{-1}\tilde A_\varepsilon} \partial_s \big(\tilde A_\varepsilon\big)^{-1} \tilde f(h^0(v^0(s)),v^0(s))\, \textnormal{d}s - h^0(v^0(t))\bigg\|_{X_\alpha}.\\
       \intertext{Integration by parts yields}
    &=       \bigg\|\textnormal{e}^{t \varepsilon^{-1} \tilde A_\varepsilon} u_0+\varepsilon^{-1}\int_0^t \textnormal{e}^{(t-s)  \varepsilon^{-1}\tilde A_\varepsilon}  \big( \tilde f(u^{\varepsilon,0}(s) ,v^0(s)) -\tilde f(h^0(v^0(s)),v^0(s)) \big) \, \textnormal{d} s\\
      &\qquad+ \big(\tilde A_\varepsilon\big)^{-1}\tilde f(h^0(v^0(t)),v^0(t))  - h^0(v^0(t)) -  \textnormal{e}^{t  \varepsilon^{-1}\tilde A_\varepsilon}\big(\tilde A_\varepsilon\big)^{-1} \tilde f(h^0(v_0),v_0) \bigg\|_{X_\alpha}.\\
          %   \bigg\|\varepsilon^{-1}\int_0^t \textnormal{e}^{(t-s)  \varepsilon^{-1}\tilde A_\varepsilon}\big(\tilde f(u^{\varepsilon,0}(s),v^0(s))-\tilde f(h^0(v^0(s)),v^0(s))\big)\, \textnormal{d} s\\
    %   &\qquad +\textnormal{e}^{t  \varepsilon^{-1}\tilde A_\varepsilon}\ \big(u_0+\textnormal{D}_x f(h^0(v^0),v^0) \big(\tilde A_\varepsilon\big)^{-1} h^0(v_0)\big)- \big(\textnormal{D}_x f(h^0(v^0),v^0)\big(\tilde A_\varepsilon\big)^{-1}+I\big) h^0(v^0(t))\bigg\|_{X_\alpha}\\
    %   & \leq C_AL_f\int_0^t \frac{\textnormal{e}^{(t-s) \varepsilon^{-1}\varepsilon^{-1}(\varepsilon\omega_A+\lambda_0)}} { \varepsilon^\gamma (t-s)^{1-\gamma}}\|(u^{\varepsilon,0}(s)-h^0(v^0(s)))\|_{X_\alpha}\, \textnormal{d}s\\
    %   &\qquad +  M_A\textnormal{e}^{\varepsilon^{-1}(\varepsilon \omega_A-\lambda_0) t}\|u_0+  \textnormal{D}_x f(h^0(v^0),v^0) \big(\tilde A_\varepsilon\big)^{-1} h^0(v_0)\|_{X_\alpha} +\| \big(\textnormal{D}_x f(h^0(v^0),v^0)\big(\tilde A_\varepsilon\big)^{-1}+I\big) h^0(v^0(t))\|_{X_\alpha}.
  \end{align*}
Recalling that $ 0 = \textnormal{D}_x f(h^0(v^0),v^0) h^0(v^0) + \tilde f(h^0(v^0),v^0)$ we estimate
\begin{align*}
    \|u^{\varepsilon,0}(t)-h^0(v^0(t))\|_{X_\alpha}&\leq C_AL_{\tilde f}\int_0^t \frac{\textnormal{e}^{(t-s) \varepsilon^{-1}(\varepsilon\omega_A+\lambda_0)}} { \varepsilon^\gamma (t-s)^{1-\gamma}}\|(u^{\varepsilon,0}(s)-h^0(v^0(s)))\|_{X_\alpha}\, \textnormal{d}s\\
    &\quad +M_A\textnormal{e}^{\varepsilon^{-1}(\varepsilon \omega_A-\lambda_0) t}\|u_0 - h^0(v_0)\|_{X_\alpha}\\
    &\quad +M_A\textnormal{e}^{\varepsilon^{-1}(\varepsilon \omega_A-\lambda_0) t}\|h^0(v_0) -  (\tilde A_\varepsilon )^{-1}\textnormal{D}_x f(h^0(v^0),v^0) h^0(v_0)\|_{X_\alpha}\\
    &\quad +\|h^0(v^0(t))- \big(\tilde A_\varepsilon\big)^{-1}\textnormal{D}_x f(h^0(v^0),v^0)  h^0(v^0(t))\|_{X_\alpha}.
\end{align*}
Here, we note the additional terms in the estimate when compared with the result in \cite[Thm. 4.13]{hummel2022slow}.
This is due to the fact that $\textnormal{D}_x f(h^0(v^0),v^0)(\tilde A_\varepsilon)^{-1}$ does not equal the identity operator $\textnormal{Id}$ but a perturbation of the latter.
To be more precise, let $z\in X_1$, then we have that
\begin{align*}
    \|\big(\textnormal{Id}- (\tilde A_\varepsilon)^{-1} \textnormal{D}_x f(h^0(v^0),v^0)\big) z\|_{X_\alpha}& = \varepsilon \|  \big(\tilde A_\varepsilon\big)^{-1}A  z\|_{X_\alpha}
\end{align*}
and we can write 
\begin{align*}
    \big(\tilde A_\varepsilon\big)^{-1} \textnormal{D}_x f(h^0(v^0),v^0)  = \textnormal{Id} -\varepsilon J,
\end{align*}
where the operator $J:=  (\tilde A_\varepsilon)^{-1} A: X_{1}\to X_\alpha$ is bounded by $\|J\|_{\mathcal{B}(X_\alpha,X_1)}\leq C \varepsilon^{-\alpha}$. Thus,
\begin{align*}
    \|u^{\varepsilon,0}(t)-h^0(v^0(t))\|_{X_\alpha}&  \leq C_AL_{\tilde f}\int_0^t \frac{\textnormal{e}^{(t-s)\varepsilon^{-1}(\varepsilon\omega_A+\lambda_0)}} { \varepsilon^\gamma (t-s)^{1-\gamma}}\|(u^{\varepsilon,0}(s)-h^0(v^0(s)))\|_{X_\alpha}\, \textnormal{d}s\\
      &\quad +  M_A \textnormal{e}^{\varepsilon^{-1}(\varepsilon\omega_A+\lambda_0) t}\|u_0- h^0(v_0)\|_{X_\alpha} + \varepsilon M_A L_{h^0} \textnormal{e}^{\varepsilon^{-1}(\varepsilon\omega_A+\lambda_0) t}\|J\|_{\mathcal{B}(X_\alpha,X_1)} \| v_0\|_{Y_1}\\
      &\quad+\varepsilon L_{h^0}\|J\|_{\mathcal{B}(X_\alpha,X_1)} \| v^0(t)\|_{Y_1}.
\end{align*}
Using estimate \eqref{ineq v0} and applying Gronwall's inequality yields
\begin{align*}
     \|u^{\varepsilon,0}(t)\!-\!h^0(v^0(t))\|_{X_\alpha}&  \leq  M_A \textnormal{e}^{\varepsilon^{-1}(\varepsilon\omega_A+\lambda_0) t}\big(\|u_0- h^0(v_0)\|_{X_\alpha} + \varepsilon^{1-\alpha} C  \| v_0\|_{Y_1}\big)\! +\varepsilon^{1-\alpha} C  \textnormal{e}^{\omega_B t}\|v_0\|_{Y_1}\\
     &\quad +C_AL_{\tilde f} M_A \|u_0- h^0(v_0)\|_{X_\alpha} \textnormal{e}^{\frac{\Gamma(\gamma)}{|\varepsilon \omega_A+ \lambda_0|^\gamma}} \int_0^t \frac{\textnormal{e}^{(t-s)\varepsilon^{-1}(\varepsilon\omega_A+\lambda_0)}} { \varepsilon^\gamma (t-s)^{1-\gamma}} \textnormal{e}^{\varepsilon^{-1}(\varepsilon\omega_A+\lambda_0) s} \,\textnormal{d}s\\
     &\quad + \varepsilon^{1-\alpha} C C_AL_{\tilde f} M_A L_h  \| v_0\|_{Y_1}  \textnormal{e}^{\frac{\Gamma(\gamma)}{|\varepsilon \omega_A+ \lambda_0|^\gamma}}  \int_0^t \frac{\textnormal{e}^{(t-s)\varepsilon^{-1}(\varepsilon\omega_A+\lambda_0)}} { \varepsilon^\gamma (t-s)^{1-\gamma}} \textnormal{e}^{\varepsilon^{-1}(\varepsilon\omega_A+\lambda_0) s} \,\textnormal{d}s\\
     &\quad+ \varepsilon^{1-\alpha} C_AL_{\tilde f}C L_h  \|v_0\|_{Y_1} \textnormal{e}^{\frac{\Gamma(\gamma)}{|\varepsilon \omega_A+ \lambda_0|^\gamma}}   \int_0^t \frac{\textnormal{e}^{(t-s)\varepsilon^{-1}(\varepsilon\omega_A+\lambda_0)}} { \varepsilon^\gamma (t-s)^{1-\gamma}}  \textnormal{e}^{\omega_B s} \,\textnormal{d}s.
\end{align*}
This inequality can be further simplified as follows
\begin{align}\label{est u3}\begin{split}
    \|u^{\varepsilon,0}(t)-h^0(v^0(t))\|_{X_\alpha}&  \leq  M_A \varepsilon^{-\gamma} \textnormal{e}^{\varepsilon^{-1}(\varepsilon\omega_A+\lambda_0) t} t^\gamma\|u_0- h^0(v_0)\|_{X_\alpha} +  \varepsilon^{1-\alpha} C \textnormal{e}^{\omega_B t} \|v_0\|_{Y_1}\\
    &\quad+ C \varepsilon^{-\gamma+(1-\alpha)}\textnormal{e}^{\varepsilon^{-1}(\varepsilon\omega_A+\lambda_0) t} t^\gamma\|v_0\|_{Y_1},\end{split}
\end{align}
where we observe that $\lim_{\varepsilon\to 0} \varepsilon^{-\gamma}\textnormal{e}^{\varepsilon^{-1}(\varepsilon\omega_A+\lambda_0) t}=0 $ for $t>0$.

Combining the three estimates for the $u$ component \eqref{est u1}-\eqref{est u3} yields
\begin{align}\label{est u all}\begin{split}
    \|u^\varepsilon(t)-h^0(v^0(t))\|_{X_\alpha}&\leq C \sup_{0\leq s\leq t}\|v^\varepsilon(s)-v^0(s)\|_{Y_1}\\
    & + C M_A\textnormal{e}^{\varepsilon^{-1}(\varepsilon\omega_A+\lambda_0) t}\|u_0- h^0(v_0)\|_{X_\alpha} +  C \varepsilon^{1-\alpha}  (\textnormal{e}^{\omega_B t}  + \varepsilon^{-\gamma}t^\gamma\textnormal{e}^{\varepsilon^{-1}(\varepsilon\omega_A+\lambda_0) t}) \|v_0\|_{Y_1},\end{split}
\end{align}
where the constants $C$ are independent of $\varepsilon$ and only depend on the nonlinearities $f$ and $g$ and on the operators $A$ and $B$.\\

Now, for the equation in $v$ we estimate
\begin{align*}
    \|v^\varepsilon(t)-v^0(t)\|_{Y_1}\leq L_g C_B \int_0^t \frac{\textnormal{e}^{(t-s)\omega_B }}{(t-s)^{1-\delta}} \big( \|u^\varepsilon(s) -h(v^0(s))\|_{X_1} +\|v^\varepsilon-v^0\|_{Y_1}\big)\, \textnormal{d} s.
\end{align*}
Using the previous estimate \eqref{est u all} for $u$ yields
\begin{align*}
     \|v^\varepsilon(t)-v^0(t)\|_{Y_1}&\leq L_g C_B \int_0^t \frac{\textnormal{e}^{(t-s)\omega_B }}{(t-s)^{1-\delta}} \bigg( \|v^\varepsilon(s)-v^0(s)\|_{Y_1} + \sup_{0\leq \tau \leq s}\|v^\varepsilon(\tau)-v^0(\tau)\|_{Y_1}\\
    &\quad+ C M_A\textnormal{e}^{\varepsilon^{-1}(\varepsilon\omega_A+\lambda_0) s}\|u_0- h^0(v_0)\|_{X_\alpha} +  C \varepsilon^{1-\alpha}  (\textnormal{e}^{\omega_B s}  + \varepsilon^{-\gamma}s^\gamma\textnormal{e}^{\varepsilon^{-1}(\varepsilon\omega_A+\lambda_0) s}) \|v_0\|_{Y_1} \bigg)\, \textnormal{d} s.
    \end{align*}
    Since the right-hand side is increasing in $t$ we have
    \begin{align*}
     \sup_{0\leq s\leq t}\|v^\varepsilon(s)-v^0(s)\|_{Y_1}&\leq C \int_0^t \frac{\textnormal{e}^{(t-s)\omega_B }}{(t-s)^{1-\delta}} \bigg(  \sup_{0\leq\tau\leq s} \|v^\varepsilon(\tau)\!-\!v^0(\tau)\|_{Y_1}\! +\! C \textnormal{e}^{\varepsilon^{-1}(\varepsilon\omega_A+\lambda_0) s}\|u_0\!-\! h^0(v_0)\|_{X_\alpha} \\
     &\quad +  C \varepsilon^{1-\alpha}  (\textnormal{e}^{\omega_B s}  + \varepsilon^{-\gamma}s^\gamma\textnormal{e}^{\varepsilon^{-1}(\varepsilon\omega_A+\lambda_0) s}) \|v_0\|_{Y_1} \bigg)\, \textnormal{d} s.
    \intertext{Applying Gronwall's inequality yields}
     \sup_{0\leq s\leq t}\|v^\varepsilon(s)-v^0(s)\|_{Y_1}&\leq C\textnormal{e}^{\textnormal{e}^{\omega_B t}} \int_0^t \frac{\textnormal{e}^{(t-s)\omega_B }}{(t-s)^{1-\delta}} \bigg(  C \textnormal{e}^{\varepsilon^{-1}(\varepsilon\omega_A+\lambda_0) s}\|u_0- h^0(v_0)\|_{X_\alpha}\\
     &\quad +  C \varepsilon^{1-\alpha}  (\textnormal{e}^{\omega_B s}  + \varepsilon^{-\gamma}s^\gamma\textnormal{e}^{\varepsilon^{-1}(\varepsilon\omega_A+\lambda_0) s}) \|v_0\|_{Y_1}\bigg)\, \textnormal{d} s. 
\end{align*}
This estimate can be further improved to
\begin{align}\label{est v all}\begin{split}
     \sup_{0\leq s\leq t}\|v^\varepsilon(s)-v^0(s)\|_{Y_1}\leq& C  {\textnormal{e}^{\textnormal{e}^{\omega_{\small{B}} t}+ct}} t^\delta   \varepsilon^{1-\alpha}  (\textnormal{e}^{\omega_B t}  + \varepsilon^{-\gamma}t^\gamma\textnormal{e}^{\varepsilon^{-1}(\varepsilon\omega_A+\lambda_0) t}) \|v_0\|_{Y_1}\\
     &+C{\textnormal{e}^{\textnormal{e}^{\omega_{\small{B}} t}+ct}} C M_A\textnormal{e}^{\varepsilon^{-1}(\varepsilon\omega_A+\lambda_0) t}\|u_0- h^0(v_0)\|_{X_\alpha}. \end{split}
\end{align}
Hence, combining the estimate for the $u$ and $v$ component yields
\begin{align*}
    \|u^\varepsilon(t)-h^0(v^0(t))\|_{X_\alpha}+\|v^\varepsilon(t)-v^0(t)\|_{Y_1} &\leq C_1\textnormal{e}^{\varepsilon^{-1}(\varepsilon\omega_A+\lambda_0) t}\|u_0- h^0(v_0)\|_{X_\alpha}\\
    &\quad+  C_2 \varepsilon^{1-\alpha}  (\textnormal{e}^{\omega_B t}  + \varepsilon^{-\gamma}t^\gamma\textnormal{e}^{\varepsilon^{-1}(\varepsilon\omega_A+\lambda_0) t}) \|v_0\|_{Y_1},
\end{align*}
which is the desired inequality.
\end{proof}

\begin{remark}
    Let us put this result into context. 
    For fixed $t\in [0,\infty)$ the difference between the solution of the fast-reaction and the limit system converges to $0$ as $\varepsilon\to 0$.
    However, when the final time $T$, which depends on the restrictions of the implicit function theorem and the on the condition $\lambda_0<0$, becomes arbitrary large the parameter $\varepsilon$ has to decrease accordingly for the convergence to still hold.
\end{remark}

\begin{remark}\label{remark 3.10}
    The loss of regularity in the convergence of trajectories is due to estimates of the operator $\tilde A_\varepsilon$, which are needed in estimates $\mathcal{II}$ and $\mathcal{III}$. There we estimate the operators $\big(\tilde A_\varepsilon\big)^{-1}$ and $A\big(\tilde A_\varepsilon\big)^{-1}$ and observe that the norm depends on $\varepsilon$ and the underlying spaces. 
    To see this, assume that $X$ is a separable Hilbert space with orthogonal basis $(e_n)_{n\in \mathbb{N}}$.
    Assume that $\tilde A_\varepsilon$ has eigenvalues $-\varepsilon\lambda_n +\lambda(t)$ such that $\sup_t \lambda(t)<\lambda_0<0$ and $\{\lambda_n\}_{n\in \mathbb{N}}$ is an unbounded decreasing sequence.
    Then, we compute for $\alpha=\beta +\mu$ 
    \begin{align*}
        \|A\big( \tilde A_\varepsilon\big)^{-1}\|_{\mathcal{B}(X_\beta,X_\alpha)}^2 &= \sup_{\|x\|_{X_\alpha}=1}\|A\big(\tilde A_\varepsilon\big)^{-1}x\|_{X_\beta}^2 \leq   \sup_{\|x\|_{X_\alpha}=1}\sum_{n\in \mathbb{N}} \frac{\lambda_n^2}{(-\varepsilon\lambda_n +\lambda(t))^2} \lambda_n^{2\beta} \langle x,e_n\rangle  \\
        &\leq \sup_{\|x\|_{X_\alpha}=1}\sum_{n\in \mathbb{N}} \frac{\lambda_n^2}{(-\varepsilon\lambda_n +\lambda_0)^2} \lambda_n^{2\beta} \langle x,e_n\rangle\\
        &\leq \sup_{\|x\|_{X_\alpha}=1}\bigg(\sum_{n=1}^{n: |\lambda_n|\leq \varepsilon^{-1}} \frac{\lambda_n^{2-2\mu} \lambda_n^{2\alpha}}{(-\varepsilon\lambda_n +\lambda_0)^2} \langle x,e_n\rangle + \sum_{n: |\lambda_n|>\varepsilon^{-1}} \frac{\lambda_n^{2-2\mu} \lambda_n^{2\alpha}}{(-\varepsilon\lambda_n +\lambda_0)^2}  \langle x,e_n\rangle\bigg)\\
        &\leq  \sup_{\|x\|_{X_\alpha}=1}\bigg(\sum_{n=1}^{n: |\lambda_n|\leq \varepsilon^{-1}} \frac{ \lambda_n^{2\alpha}}{(-\varepsilon\lambda_n^\mu +\lambda_0 \lambda^{2\mu-2})^2} \langle x,e_n\rangle + \sum_{n: |\lambda_n|>\varepsilon^{-1}} \frac{\lambda_n^{2-2\mu} \lambda_n^{2\alpha}}{(-\varepsilon\lambda_n +\lambda_0 )^2}   \langle x,e_n\rangle\bigg)\\
        &\leq \sup_{\|x\|_{X_\alpha}=1}\bigg(\sum_{n=1}^{n: |\lambda_n|\leq \varepsilon^{-1}} \frac{ \lambda_n^{2\alpha}}{(\lambda_0 \lambda^{2\mu-2})^2} \langle x,e_n\rangle + \sum_{n: |\lambda_n|>\varepsilon^{-1}} \frac{\varepsilon^{2\mu} \lambda_n^{2\alpha}}{\varepsilon^{2} }  \langle x,e_n\rangle\bigg)\\
        &\leq C\|x\|_{X_\alpha}^2 + \varepsilon^{-2(1-\mu)}\|x\|_{X_\alpha} ^2\leq \varepsilon^{-2(1-\mu)} C \|x\|_{X_\alpha}^2
    \end{align*}
    Thus, only for $\alpha-\beta<1$ the norm of $ \varepsilon \|A\big( \tilde A_\varepsilon\big)^{-1}\|_{\mathcal{B}(X_\beta,X_\alpha)}$ can be controlled by $\varepsilon^\mu$, with $\mu>0$.
    Therefore, we require the estimates are posed in $X_\alpha$ with $0\leq \alpha<1$, whereas the initial data is in $X_1$.
    This marks also a further difference to the results in \cite{hummel2022slow}.
\end{remark}

\section{Slow Manifolds}\label{section 4}

In this section we generalize the Fenichel theory for ODEs to infinite dimensional systems.
This theory originates from the works of Fenichel as a perturbation theory of normally hyperbolic invariant manifolds \cite{fenichel1971persistence} and provides a a suitable framework for the treatment of fast-slow ODE systems \cite{fenichel1979geometric,jones1995geometric,kuehn2015multiple}

Following the structure of the Fenichel theorem in finite dimension we establish the existence of a slow manifold, show that the distance between this slow manifold and the critical manifold is small in a suitable sense and that the slow flow on the critical manifold is a good approximation of the flow on the slow manifold.

One important aspect of the Fenichel theory is the notion of invariant manifold.
A manifold $M$ is forward invariant with respect to the semigroup $(T(t))_{t\geq 0}$ if for all $m\in M$ and all $t\geq 0$ it holds that $T(t) m\in M$.

\begin{theorem}[General Fenichel Theorem for infinite-dimensional fast-reaction systems]\label{general fenichel thm}
Suppose $S_{0,\zeta}$ is a submanifold of the critical manifold $S_0$ and let $\tilde A_\varepsilon$ be the dissipative operator of the previous section introduced in system \eqref{modified fast reaction}. 
In addition, let the technical assumptions of Sections \ref{sect 3.1} and \ref{sect 4.1} hold. 
Then, for $\varepsilon,\,\zeta>0$ small enough the following hold
\begin{itemize}
    \item[(GF1)] There exists a locally forward invariant manifold $S_\varepsilon$ for the system \eqref{full eq splitting} called the slow manifold. Local invariance means that trajectories can enter or leave $S_\varepsilon$ only through its boundaries;
    \item[(GF2)] The slow manifold $S_\varepsilon$ has a distance of $\mathcal{O}(\varepsilon)$ to $S_0$ in the $Y_1$-norm, i.e. the norm generated by the operator $B$;
    \item[(GF3)] The slow manifold $S_\varepsilon$ is $C^1$-regular in the topology induced by the operators $A$ and $B$;
    \item[(GF4)] The flow on $S_\varepsilon$ converges to the semi-flow on $S_{0,\zeta}$, where again the convergence is with respect to the $Y_1$-norm.
\end{itemize}
\end{theorem}
The details as well as the proofs of each of the statements follow in  the subsequent subsections.

\begin{remark}
    We want to remark on the difference between the results presented in this section and the result by Bates et al. in  \cite{bates1998existence}.
    One key assumption in the later work is that 
    \begin{align*}
        \| T_{\varepsilon}(t)-T_0(t)\|_{C^1(N, X\times Y)} \to 0 \quad\text{as } \varepsilon\to 0,
    \end{align*}
    where $N$ is a neighborhood of $S_0$.
    When we compare this assumption with Theorem \ref{Thm 3.6} we observe that we only require
    \begin{align*}
        \| T_{\varepsilon}(t)-T_0(t)\|_{C(N, X_1\times Y_1)} \to 0 \quad\text{as } \varepsilon \to 0
    \end{align*}
    for $t\in [0,\infty)$ to hold.
    Hence, our result applies to a larger class of fast-reaction systems.
    One way to satisfy the assumption in \cite{bates1998existence} is to assume that the operator $B$ generates a $C^0$-group.
    We, however, choose a different approach in extending the Fenichel theory to infinite dimensions.
    Lastly, in the case of the equation for the slow component is an ODE or if $B$ generates a $C^0$-group we observe that the results of Theorem \ref{general fenichel thm} coincide with the ones in \cite{bates1998existence}. 
\end{remark}

\subsection{Assumptions}\label{sect 4.1}

To show the results in this section we need additional assumptions on the operator $B$ and the Banach space $Y$.
The reason for this is that we want to express all the terms that arise in the upcoming estimates in terms of the slow variable $v$.
However, when we rewrite the fast reaction system in the slow time scale by $\tau =\varepsilon t$
\begin{align*}\begin{split}
    \partial_\tau u^\varepsilon &= \tilde A_\varepsilon u^\varepsilon + \tilde f(u^\varepsilon,v^\varepsilon),\\
    \partial_\tau v^\varepsilon&= \varepsilon \big( Bv^\varepsilon + g(u^\varepsilon,v^\varepsilon) \big)\end{split}
\end{align*}
we observe that in the limit $\varepsilon\to 0$ the expression $\varepsilon B v^\varepsilon$ might not be well-defined as $B$ is an unbounded operator. For more details and an example we refer to \cite[Section 3.3]{hummel2022slow}.

Therefore, we introduce a splitting in the slow variable space
\begin{align*}
    Y= Y^\zeta_F\oplus Y^\zeta_S
\end{align*}
into a fast part $Y^\zeta_F$ and a slow part $Y^\zeta_S$, where  $\zeta>0$ is a small parameter, such that 
\begin{itemize}
    \item The spaces $Y^\zeta_F$ and $Y^\zeta_S$ are closed in $Y$ and the projections $\pr_{Y^\zeta_F}$ and $\pr_{Y^\zeta_S}$ commute with $B$ on $Y_1$.
    \item The spaces $Y^\zeta_F\cap Y_1$ and $Y^\zeta_S\cap Y_1$ are closed subspaces of $Y_1$ and are endowed with the norm $\|\cdot\|_{Y_1}$.
    \item The realization of $B$ in $Y^\zeta_F$, i.e. 
    \begin{align*}
        B_{Y^\zeta_F}:D(Y^\zeta_F)&\subset Y^\zeta_F\to Y^\zeta_F,\quad v\mapsto Bv
        \intertext{with}
        D(Y^\zeta_F)&:= \{v_0\in Y^\zeta_F\,:\, Bv_0\in Y^\zeta_F\}
    \end{align*}
    has $0$ in its resolvent set.
    \item The realization of $B$ in $Y^\zeta_S$ generates a $C_0$-group $(\textnormal{e}^{tB_{Y^\zeta_S}})_{t\in \mathbb{R}}\subset \mathcal{B}((Y^\zeta_S,\|\cdot\|_{Y}) )$ which for $t\geq 0$ satisfies $\textnormal{e}^{tB_{Y^\zeta_S}}=\textnormal{e}^{tB}$ on $Y^\zeta_S$.
    \item The fast subspace $Y^\zeta_F$ contains the parts of $Y_1$ that decay under the semigroup $(\textnormal{e}^{tB})_{t\geq 0}$ almost as fast as functions in the fast variable space $X_1$ under the semi-group $(\textnormal{e}^{\zeta^{-1}t\tilde A_\varepsilon})_{t\geq 0}$ generated by the modified operator $\tilde A_\varepsilon$.
    The space $Y^\zeta_S$ on the other hand contains the parts of $Y_1$ that do not decay or which decay only slowly under the semigroup $(\textnormal{e}^{tB})_{t\geq 0}$ compared to $X_1$ under $(\textnormal{e}^{\zeta^{-1}t \tilde A_\varepsilon})_{t\geq 0}$. 
    Hence, there are constants $C_B,\,M_B>0$ such that for $\zeta>0$ small enough there are constants $N_F^\zeta,N_S^\zeta$ satisfying   
    $$0\leq N_F^\zeta< N_S^\zeta\leq |\zeta^{-1}(\varepsilon\omega_A+\lambda_0)|$$
    such that for all $t\geq 0$ and $y_F\in Y^\zeta_F,\, y_S\in Y^\zeta_S$ we have 
    \begin{align*}
        \|\textnormal{e}^{tB}y_F\|_{Y_1}&\leq C_B t^{\delta -1} \textnormal{e}^{(N_F^\zeta +\zeta^{-1}(\varepsilon\omega_A+\lambda_0))t}\|y_F\|_{Y_\delta},\\
        \|\textnormal{e}^{-tB}y_S\|_{Y_1}&\leq M_B \textnormal{e}^{-(N_S^\zeta +\zeta^{-1}(\varepsilon\omega_A+\lambda_0))t}\|y_S\|_{Y_1}.
    \end{align*}
    \item A sufficient condition on the size of the parameter $\zeta>0$ is the relation  $\varepsilon \zeta^{-1}\leq c <1$.
    This assures that the following inequality holds true
         $$(1-\varepsilon \zeta^{-1})(\varepsilon\omega_A+\lambda_0)-\frac{\varepsilon}{2}(N_S^\zeta+N_F^\zeta)<0.$$
        
    \item Lastly, the following spectral gap condition holds
    \begin{align}\label{spectral gap condition}
         L:=  \frac{2^\gamma L_{\tilde f} C_A \Gamma(\gamma)}{(2(\varepsilon\zeta^{-1} -1)(\varepsilon\omega_A+\lambda_0) + \varepsilon (N_S^\zeta+N_F^\zeta) )^\gamma} + \frac{2^\delta L_gC_B \Gamma(\delta)}{(N_S^\zeta-N_F^\zeta)^\delta}+\frac{2L_g M_B \Gamma(\delta)}{N_S^\zeta-N_F^\zeta}< 1
    \end{align}
\end{itemize}
\begin{remark}
Here, we take a closer look at the terms arising in the spectral gap condition \eqref{spectral gap condition}. 
As $\varepsilon \zeta^{-1}\leq c <1$ holds and thus $(1-\varepsilon \zeta^{-1})(\varepsilon \omega_A+\lambda_0)>$ the first term can always be controlled by a constant $C<1$.\\
The quantities $N_S^\zeta,\, N^\zeta_F$ can be seen as shifted eigenvalues of the operator $\tilde A_\varepsilon$ and therefore $N_S^\zeta-N^\zeta_F$  can be considered as the spectral gap of this operator.
It relates to the parameter $\zeta$ through $N_S^\zeta-N^\zeta_F = C(\zeta)$, where $C(\zeta)\to \infty$ as $\zeta\to 0$, (see Subsection \ref{subsection example} for an example).\\
However, this spectral gap condition may be satisfied in low dimensions only.
Recalling that the Weyl asymptotic for the eigenvalues of the Laplacian reads $\lambda_m \sim C m^{2/d}$ yields for a domain $\Omega\subset\mathbb{R}$ that $N_S^\zeta-N^\zeta_F =\mathcal{O}( \zeta^{-1/2})$. 
For a general domain $\Omega\subset \mathbb{R}^d$, $d\geq 2$ this asymptotic may fail. Yet, for special domains such as $\Omega=[0,2\pi]^2$ a number theoretical result from \cite{richards1982gaps} yields that $N_S^\zeta-N^\zeta_F = \mathcal{O}(\log(\zeta))$.
Methods to overcome this problem are discussed in \cite{mallet1988inertial,zelik2014inertial,kostianko2021kwak}.
\end{remark}
With this splitting we can rewrite the fast reaction system in the following way
\begin{align}\label{full eq splitting}
    \begin{split}
             \partial_t u^\varepsilon(t)&=\frac{1}{\varepsilon} \tilde A_\varepsilon u^\varepsilon(t)+ \tilde f(u^\varepsilon(t),v_F^\varepsilon(t),v_S^\varepsilon(t)),\\
           \partial_t v^\varepsilon_F(t)&=B v^\varepsilon_F(t)+ \pr_{Y_F^\zeta} g(u^\varepsilon(t),v_F^\varepsilon(t),v_S^\varepsilon(t)),\\
            \partial_t v^\varepsilon_S(t)&=B v^\varepsilon_S(t)+ \pr_{Y_S^\zeta} g(u^\varepsilon(t),v_F^\varepsilon(t),v_S^\varepsilon(t)),\\
            u^\varepsilon(0)&=u_0,\quad     v^\varepsilon_F(0)=\pr_{Y_F^\zeta} v_0,\quad ~\, v^\varepsilon_S(0)=\pr_{Y_S^\zeta} v_0.
    \end{split}
\end{align}
% \begin{remark}
%     As in the previous section we will work with the equivalent system  
%   \begin{align*}
%       \varepsilon \partial_t u^\varepsilon(t)&=\frac{1}{\varepsilon}\big( \tilde A_\varepsilon u^\varepsilon(t)+ \tilde f(u^\varepsilon(t),v_F^\varepsilon(t),v_S^\varepsilon(t))\big) ,\\
%           \partial_t v^\varepsilon(t)&=B v^\varepsilon(t)+ \pr_{Y_F^\zeta} g(u^\varepsilon(t),v_F^\varepsilon(t),v_S^\varepsilon(t)),\\
%             \partial_t v^\varepsilon(t)&=B v^\varepsilon(t)+ \pr_{Y_S^\zeta} g(u^\varepsilon(t),v_F^\varepsilon(t),v_S^\varepsilon(t)),
%   \end{align*} 
   
% \end{remark}

\subsection{Existence of slow manifolds}
We construct a family of slow manifolds $S_{\varepsilon,\zeta}$ as 
\begin{align}
    S_{\varepsilon,\zeta}:=\{(h^{\varepsilon,\zeta}(v_0),v_0):\, v_0\in Y_1\cap Y_S^\zeta\},
\end{align}
where the function $h^{\varepsilon,\zeta}$ maps
\begin{align*}
    h^{\varepsilon,\xi}: (Y_S^\zeta\cap Y_1)\to  X_1
    \times(Y_F^\zeta\cap Y_1).
\end{align*}
The next step is to construct the functions $h^{\varepsilon,\zeta}$ using a Lyapunov-Perron method (cf. \cite{chow1991smooth,henry2006geometric}).
To this end, we introduce the operator
\begin{align}\label{eq: Lyapunov-Perron}
\begin{split}
    \mathcal{L}_{v_0,\varepsilon,\zeta}:C_\eta&\to C_\eta,\\
    \begin{pmatrix}
    u\\v_F\\v_S
    \end{pmatrix}&\mapsto \left[ t\mapsto \begin{pmatrix}
    \varepsilon^{-1}\int_{-\infty}^t \textnormal{e}^{(t-s)\tilde A_\varepsilon} \tilde f(u(s),v_F(s),v_S(s))\, \textnormal{d}s\\ \int_{-\infty}^t \textnormal{e}^{(t-s)B}\pr_{Y^\zeta_F}g(u(s),v_F(s),v_S(s))\, \textnormal{d}s\\ \textnormal{e}^{tB}v_0+\int_{0}^t \textnormal{e}^{(t-s)B}\pr_{X^\zeta_S}g(u(s),v_F(s),v_S(s))\, \textnormal{d}s
    \end{pmatrix}\right],
    \end{split}
\end{align}
 where $v_0\in Y_1\cap Y^\zeta_S$ and 
 \[C_\eta:= C((-\infty,0],\textnormal{e}^{\eta t}; X_1
    \times(Y_F^\zeta\cap Y_1)\times (Y_1\cap Y^\zeta_S)) \]
    is the space of $(u,v_F,v_S)\in C((-\infty,0]; X_1
    \times(Y_F^\zeta\cap Y_1)\times ( Y^\zeta_S \cap Y_1) $ such that
    \begin{align*}
        \|(u,v_F,v_S)\|_{C_\eta}:=\sup_{t\leq 0} \textnormal{e}^{-\eta t}(\|u\|_{X_1}+\|v_F\|_{Y_1}+\|v_S\|_{Y_1})<\infty.
    \end{align*}

\begin{prop} \label{Prop 4.2}
    Let $v_0\in Y_1\cap Y^\zeta_S$. Then the operator $\mathcal{L}_{v_0,\varepsilon,\zeta}$ has a unique fixed point in $C_\eta$.
\end{prop}
% \begin{remark}
%     The function $h^{\varepsilon,\zeta}$ used to construct the family of slow manifolds is given by the first two components of the fixed point of $\mathcal{L}_{v_0,\varepsilon,\zeta}$ evaluated at $t=0$,
%     i.e. $h^{\varepsilon,\zeta}(v_0)= (u^{v_0}(0),v^{v_0}_F(0))$.
% \end{remark}

\begin{proof}
As a first step we rewrite the system into the modified fast-slow system while maintaining the splitting in the slow variable $v$, where we replace the operator $A$ by $\frac{1}{\varepsilon}\tilde A_\varepsilon$ and the nonlinear function $f$ by $\tilde f$.
Then, we show that the Lyapunov-Perron operator introduced in equation \eqref{eq: Lyapunov-Perron} has a unique fixed point.
To do this, we apply the existence result shown in \cite[Prop. 5.1]{hummel2022slow}, where we set
\begin{align*}
    \eta= \zeta^{-1}(\varepsilon\omega_A+\lambda_0) +\frac{N_S^\zeta+N_F^\zeta}{2}.
\end{align*}
The function $h^{\varepsilon,\zeta}$ used to construct the family of slow manifolds is given by the first two components of the fixed point of the Lyapunov-Perron operator $\mathcal{L}_{v_0,\varepsilon,\zeta}$ evaluated at $t=0$, i.e. $h^{\varepsilon,\zeta}(v_0)= (u^{v_0}(0),v^{v_0}_F(0))$.
This completes the first part of the statement (GF1) in Theorem \ref{general fenichel thm}.
In addition, the invariance of the slow manifold follows from the construction of the Lyapunov-Perron operator, where we consider all solutions in backward time that can develop from the initial condition in the slow component.
\end{proof}

\begin{prop}
Let $(u^{v_0},v_F^{v_0},v_S{v_0})$ be the unique fixed point of $\mathcal{L}_{v_0,\varepsilon,\zeta}$ evaluated at $t=0$ and assume that 
\begin{align*}
  L:=  \frac{2^\gamma L_{\tilde f} C_A \Gamma(\gamma)}{(2(\varepsilon\zeta^{-1} -1)(\varepsilon\omega_A+\lambda_0) + \varepsilon (N_S^\zeta+N_F^\zeta) )^\gamma} + \frac{2^\delta L_gC_B \Gamma(\delta)}{(N_S^\zeta-N_F^\zeta)^\delta}+\frac{2L_gM_B \Gamma(\delta)}{N_S^\zeta-N_F^\zeta}< 1.
\end{align*}
Then, the mapping
\begin{align*}
    h^{\varepsilon,\zeta}: (Y^\zeta_S\cap Y_1)\to X_1\times (Y^\zeta_F\cap Y_1),\, v_0\mapsto (u^{v_0}(0), v_F^{v_0}(0))
\end{align*}
is Lipschitz continuous with Lipschitz constant $L_\zeta= M_B/(1-L)$.
\end{prop}
\begin{proof}
This follows from the estimates in Proposition \ref{Prop 4.2}, cf. \cite[Prop 5.2]{hummel2022slow}.
\end{proof}

\subsection{Distance to critical manifold}
\begin{prop}\label{Prop 4.6}
There exists a constant $C>0$ depending on the operator $A$ and the nonlinear functions $f$ and $g$ such that for all $\varepsilon,\,\zeta>0$ small enough satisfying $\varepsilon\zeta^{-1}<1$ and all initial data $v_0 \in Y_S^\zeta$ it holds that 
\begin{align}
    \bigg\|\begin{pmatrix}
   h^{\varepsilon,\zeta}_{X_1}(v_0)- h^0(v_0)\\ h^{\varepsilon,\zeta}_{Y_F}(v_0)
    \end{pmatrix}\bigg\|_{X_\alpha\times Y_1}&\leq C\bigg( \varepsilon^{1-\alpha} +\frac{1}{(N_S^\zeta-N_F^\zeta)^\delta}\bigg) \|v_0\|_{Y_1}.
\end{align}
\end{prop}
\begin{proof}
Let $(\bar u, \bar v_F, \bar v_S)\in C_\eta$ be the unique fixed point of $\mathcal{L}_{v_0,\varepsilon,\zeta}$.
% and by extending the definition of $h^{\varepsilon,\zeta}$ for $t\leq 0$ 
That is we have 
\[(\bar u, \bar v_F, \bar v_S)=(h^{\varepsilon,\zeta}_{X_1}(\bar v_S),h^{\varepsilon,\zeta}_{Y_F}(\bar v_S)), \bar v_S).\]
Since $(\bar u, \bar v_F, \bar v_S)$ solves the system \eqref{full eq splitting} on $(-\infty,0]$ we have that $\bar v_S\in C^1((-\infty,0],\textnormal{e}^{\eta t};Y)$.
We estimate that
\begin{align*}
  \|h^{\varepsilon,\zeta}_{Y_F}(\bar v_S(t))\|_{Y_1}& = \bigg\|\int_{-\infty}^t \textnormal{e}^{(t-s)B}\pr_{Y_F^\zeta} g(h^{\varepsilon,\zeta}_{X_1}(\bar v_S(s)),h^{\varepsilon,\zeta}_{Y_F}(\bar v_S(s)),\bar v_S(s)) \, \textnormal{d} s\bigg\|_{Y_1}\\
  &\leq L_gC_B \textnormal{e}^{\eta t} \|h^{\varepsilon,\zeta}_{X_1}(\bar v_S(s)),h^{\varepsilon,\zeta}_{Y_F}(\bar v_S(s)),\bar v_S(s)\|_{C_\eta}\int_{-\infty}^t \frac{\textnormal{e}^{(t-s)(\zeta^{-1}\omega_A +N_F^\zeta -\eta)}}{(t-s)^{1-\delta}}\, \textnormal{d}s\\
  &\leq \frac{L_\zeta L_g C_B \Gamma(\delta)\textnormal{e}^{\eta t}}{ (\eta -\zeta^{-1}(\varepsilon\omega_A+\lambda_0)- N_F^\zeta)^\delta}\|v_0\|_{Y_1}\\
  &\leq C \frac{1}{N_S^\zeta-N_F^\zeta} \|v_0\|_{Y_1} .
\end{align*}
In addition, for $t_0\leq t\leq 0$ we compute
\begin{align*}
    h^{\varepsilon,\zeta}_{X_1}(\bar v_S(t))- h^0(\bar v_S(t))&=  \varepsilon^{-1}\int_{-\infty}^t \textnormal{e}^{(t-s)\varepsilon^{-1} \tilde A_\varepsilon}  \tilde f(h^{\varepsilon,\zeta}_{X_1}(\bar v_S(s)),h^{\varepsilon,\zeta}_{Y_F}(\bar v_S(s)),\bar v_S(s)) \, \textnormal{d} s - h^0(\bar v_S(t)) \\
    &=  \varepsilon^{-1}\int_{-\infty}^t \textnormal{e}^{(t-s)\varepsilon^{-1} \tilde A_\varepsilon}  \tilde f(h^{\varepsilon,\zeta}_{X_1}(\bar v_S(s)),h^{\varepsilon,\zeta}_{Y_F}(\bar v_S(s)),\bar v_S(s)) \, \textnormal{d} s\\
    &\quad  +\big(\textnormal{D}_x f(h^0(v^0),v^0)\big)^{-1} \tilde f( h^0(\bar v_S),0,\bar v_S) -\big( \tilde A_\varepsilon)^{-1} \tilde f( h^0(\bar v_S),0,\bar v_S)\\
    &\quad+ \big( \tilde A_\varepsilon)^{-1} \tilde f( h^0(\bar v_S),0,\bar v_S)\\
    \intertext{Using integration by parts yields}
    &= \varepsilon^{-1}\int_{-\infty}^{t_0} \textnormal{e}^{(t-s)\varepsilon^{-1} \tilde A_\varepsilon}  \tilde f(h^{\varepsilon,\zeta}_{X_1}(\bar v_S(s)),h^{\varepsilon,\zeta}_{Y_F}(\bar v_S(s)),\bar v_S(s)) \, \textnormal{d} s\\
    &\quad +\int_{t_0}^t  \textnormal{e}^{(t-s)\varepsilon^{-1} \tilde A_\varepsilon}\big(\tilde A_\varepsilon\big)^{-1}\partial_s \tilde f( h^0(\bar v_S(s)),0,\bar v_S(s))\, \textnormal{d}s  \\
    &\quad + \varepsilon^{-1}\int_{t_0}^{t} \textnormal{e}^{(t-s)\varepsilon^{-1} \tilde A_\varepsilon} \big( \tilde f(h^{\varepsilon,\zeta}_{X_1}(\bar v_S(s)),h^{\varepsilon,\zeta}_{Y_F}(\bar v_S(s)),\bar v_S(s))- \tilde f( h^0(\bar v_S(s)),0,\bar v_S(s))\big)\, \textnormal{d} s\\
    &\quad +\textnormal{e}^{(t-t_0)\varepsilon^{-1} \tilde A_\varepsilon}\big(\tilde A_\varepsilon\big)^{-1} \tilde f( h^0(\bar v_S(t_0)),0,\bar v_S(t_0))\\
    &\quad+\big((\textnormal{D}_x f(h^0(v^0),v^0))^{-1}- (\tilde A_\varepsilon)^{-1}\big) \tilde f( h^0(\bar v_S),0,\bar v_S). 
\end{align*}
Again we observe the additional terms in the last line due to the modified operator $\tilde A_\varepsilon$, which will lead again to a lower regularity of the estimate, when comparing this proof with the one in \cite{hummel2022slow}.
Hence, we estimate the above expression as follows
\begin{align*}
   \|h^{\varepsilon,\zeta}_{X_1}(\bar v_S)&- h^0(\bar v_S)\|_{X_\alpha} \leq \\
    & \leq C_A L_{\tilde f} \int_{-\infty}^{t_0} \frac{\textnormal{e}^{(t-s)\varepsilon^{-1} (\varepsilon\omega_A+\lambda_0)}}{\varepsilon^\gamma (t-s)^{1-\gamma}} \|h^{\varepsilon,\zeta}_{X_1}(\bar v_S(s)),h^{\varepsilon,\zeta}_{Y_F}(\bar v_S(s)),\bar v_S(s))\|_{X_1\times Y_1\times Y_1} \, \textnormal{d} s \\
    &\quad + \varepsilon \|\big(\tilde A_\varepsilon\big)^{-1}\|_{\mathcal{B}(X,X_{\gamma-(1-\alpha)})}\int_{t_0}^t \frac{\textnormal{e}^{(t-s)\varepsilon^{-1} (\varepsilon\omega_A+\lambda_0)}}{\varepsilon^\gamma (t-s)^{1-\gamma}} \, \textnormal{d} s \|\tilde f (h^0(\bar v_S),\bar v_S)\|_{L^\infty((-\infty,0];X)}\\
    &\quad+ C_A L_{\tilde f} \int_{t_0}^t \frac{\textnormal{e}^{(t-s)\varepsilon^{-1} (\varepsilon\omega_A+\lambda_0)}}{\varepsilon^\gamma (t-s)^{1-\gamma}} \|h^{\varepsilon,\zeta}_{X_1}(\bar v_S(s))-h^0(\bar v_S(s)),h^{\varepsilon,\zeta}_{Y_F}(\bar v_S(s))\|_{X_\alpha\times Y_1}\, \textnormal{d}s\\
    &\quad+ M_A L_{\tilde f} \textnormal{e}^{\varepsilon^{-1}(\varepsilon\omega_A+\lambda_0)(t-t_0)}  \|\big(\tilde A_\varepsilon\big)^{-1}\|_{\mathcal{B}(X_\gamma,X_1)} \| h^0(\bar v_S(t_0)),\bar v_S(t_0)\|_{X_1\times Y_1}\\
    &\quad +L_{\tilde f}\varepsilon \| \big(\textnormal{D}_x f(h^0(v^0),v^0)\big)^{-1}J\|_{\mathcal{B}(X_\gamma-(1-\alpha),X_1)}\| h^0(\bar v_S(t)),\bar v_S(t)\|_{X_1\times Y_1}\\
    & \leq  C_A L_{\tilde f}\textnormal{e}^{(t-t_0)\varepsilon^{-1} (\varepsilon\omega_A+\lambda_0)+\eta t_0}  \int_{-\infty}^{t_0} \frac{\textnormal{e}^{(t_0-s)\varepsilon^{-1} (\varepsilon\omega_A+\lambda_0-\varepsilon\eta )}}{\varepsilon^\gamma (t-s)^{1-\gamma}}\, \textnormal{d} s  \|h^{\varepsilon,\zeta}_{X_1}(\bar v_S),h^{\varepsilon,\zeta}_{Y_F}(\bar v_S),\bar v_S)\|_{C_\eta}\\
    &\quad + \varepsilon L_{\tilde f}\|\big(\tilde A_\varepsilon\big)^{-1}\|_{\mathcal{B}(X,X_{\gamma-(1-\alpha)})}\int_{t_0}^t \frac{\textnormal{e}^{(t-s)\varepsilon^{-1} (\varepsilon\omega_A+\lambda_0)}}{\varepsilon^\gamma (t-s)^{1-\gamma}} \, \textnormal{d} s \times\\
    &\qquad\times (1+L_h)\tilde C \textnormal{e}^{\omega_B t} \big(\|Bv_0\|_Y+ L_g(L_h+1)\|v_0\|_Y + \|v_0\|_{Y_1}\big)\\
    &\quad+ C_A L_{\tilde f} \int_{t_0}^t \frac{\textnormal{e}^{(t-s)\varepsilon^{-1} (\varepsilon\omega_A+\lambda_0)}}{\varepsilon^\gamma (t-s)^{1-\gamma}} \|h^{\varepsilon,\zeta}_{X_1}(\bar v_S(s))-h^0(\bar v_S(s))\|_{X_\alpha}\, \textnormal{d}s\\
    &\quad  + C_A L_{\tilde f} \frac{L_\zeta L_g C_B \Gamma(\delta)}{ (\eta -\zeta^{-1}( \varepsilon\omega_A+\lambda_0)- N_F^\zeta)^\delta}\int_{t_0}^t \frac{\textnormal{e}^{(t-s)\varepsilon^{-1} (\varepsilon\omega_A+\lambda_0)}}{\varepsilon^\gamma (t-s)^{1-\gamma}} \textnormal{e}^{\eta s} \, \textnormal{d}s \|v_0\|_{Y_1}\\
    &\quad +M_A L_{\tilde f} \textnormal{e}^{\varepsilon^{-1}(\varepsilon\omega_A+\lambda_0)(t-t_0)+\eta t_0}  \|\big(\tilde A_\varepsilon\big)^{-1}\|_{\mathcal{B}(X_\gamma,X_1)} \| h^0(\bar v_S),0,\bar v_S\|_{C_\eta}\\
    &\quad +L_{\tilde f} \lambda_0^{-1}\varepsilon \|J\|_{\mathcal{B}(X_\alpha,X_1)} (L_h+1)\| \bar v_S(t)\|_{ Y_1}.
    \intertext{Using the Lipschitz continuity of $h^{\varepsilon,\zeta}$ and the observation of Remark \ref{remark 3.10} yields}
    &\leq C_A L_{\tilde f} \frac{\textnormal{e}^{t_0\varepsilon^{-1} (\varepsilon\eta -(\varepsilon\omega_A+\lambda_0))} \textnormal{e}^{t\varepsilon^{-1} (\varepsilon\omega_A+\lambda_0) } }{-\omega_A +\varepsilon^{-1}\lambda_0 +\eta} L_\zeta\|v_0\|_{Y_1}\\
    &\quad + \varepsilon^{2-\gamma-\alpha}C L_{\tilde f}\frac{\Gamma(\gamma)}{|\varepsilon\omega_A -\lambda_0|^{\gamma}}(1+L_h)\tilde C \textnormal{e}^{\omega_B t} \big(\|Bv_0\|_Y+ L_g(L_h+1)\|v_0\|_Y + \|v_0\|_{Y_1}\big)\\
    &\quad  + C_A L_{\tilde f} \frac{L_\zeta L_g C_B \Gamma(\delta)}{ (\eta -\zeta^{-1}( \varepsilon\omega_A+\lambda_0)- N_F^\zeta)^\delta} \frac{\Gamma(\gamma)}{|\varepsilon\omega_A -\lambda_0|^{\gamma}} \textnormal{e}^{\eta t}\|v_0\|_{Y_1}\\
    &\quad +M_A L_{\tilde f} L_\zeta \varepsilon^{\gamma-1} C\textnormal{e}^{t_0\varepsilon^{-1}(\varepsilon\eta -(\varepsilon\omega_A+\lambda_0))}  \textnormal{e}^{t \varepsilon^{-1}(\varepsilon\omega_A+\lambda_0)} \| v_0\|_{Y_1}\\
    &\quad +L_{\tilde f} \lambda_0^{-1}\varepsilon^{1-\alpha} C (L_h+1)M_B \textnormal{e}^{\omega_B t} \textnormal{e}^{C_BL_g(1+L_h)\Gamma(\delta)} \|v_0\|_{Y_1}\\
    &\quad+ C_A L_{\tilde f} \int_{t_0}^t \frac{\textnormal{e}^{(t-s)\varepsilon^{-1} (\varepsilon\omega_A+\lambda_0)}}{\varepsilon^\gamma (t-s)^{1-\gamma}} \|h^{\varepsilon,\zeta}_{X_1}(\bar v_S(s))-h^0(\bar v_S(s))\|_{X_\alpha}\, \textnormal{d}s.
\end{align*}
We observe that the right-hand side of the inequality is of the form
\begin{align*}
     \|h^{\varepsilon,\zeta}_{X_1}(\bar v_S)&- h^0(\bar v_S)\|_{X_\alpha} \leq R(t) + C_A L_{\tilde f} \int_{t_0}^t \frac{\textnormal{e}^{(t-s)\varepsilon^{-1} (\varepsilon\omega_A+\lambda_0)}}{\varepsilon^\gamma (t-s)^{1-\gamma}} \|h^{\varepsilon,\zeta}_{X_1}(\bar v_S(s))-h^0(\bar v_S(s))\|_{X_\alpha}\, \textnormal{d}s,
\end{align*}
  where $R(t)$ is non-decreasing in time and thus
 \begin{align*}
     \|h^{\varepsilon,\zeta}_{X_1}(\bar v_S(t))&- h^0(\bar v_S(t))\|_{X_\alpha}  \leq C_A L_{\tilde f} \frac{\textnormal{e}^{t_0\varepsilon^{-1} (\varepsilon\eta -(\varepsilon\omega_A+\lambda_0))} }{-\omega_A +\varepsilon^{-1}\lambda_0 +\eta} L_\zeta\|v_0\|_{Y_1}\\
    &\quad + \varepsilon^{2-\gamma-\alpha} C L_{\tilde f}\frac{\Gamma(\gamma)}{|\varepsilon\omega_A -\lambda_0|^{\gamma}}(1+L_h)\tilde C \big(\|Bv_0\|_Y+ L_g(L_h+1)\|v_0\|_Y + \|v_0\|_{Y_1}\big)\\
    &\quad  + C_A L_{\tilde f} \frac{L_\zeta L_g C_B \Gamma(\delta)}{ (\eta -\zeta^{-1}( \varepsilon\omega_A+\lambda_0)- N_F^\zeta)^\delta} \frac{\Gamma(\gamma)}{|\varepsilon\omega_A -\lambda_0|^{\gamma}} \|v_0\|_{Y_1}\\
    &\quad +M_A L_{\tilde f}  L_\zeta \varepsilon^{\gamma-1}\textnormal{e}^{t_0\varepsilon^{-1}(\varepsilon\eta -(\varepsilon\omega_A+\lambda_0))}  \| v_0\|_{Y_1}\\
    &\quad +L_{\tilde f} \lambda_0^{-1}\varepsilon^{1-\alpha} C (L_h+1)M_B \textnormal{e}^{C_BL_g(1+L_h)\Gamma(\delta)} \|v_0\|_{Y_1}.
 \end{align*}
Using the assumption that $\varepsilon \zeta^{-1}<1$ we have that $\varepsilon\eta -(\varepsilon\omega_A+\lambda_0)>0$. Hence, we can let $t_0\to -\infty$ and obtain
\begin{align*}
    C_1^{-1}  \|h^{\varepsilon,\zeta}_{X_1}(\bar v_S(t))&- h^0(\bar v_S(t))\|_{X_1}  \leq \\
    & \leq  \varepsilon^{2-\gamma-\alpha} L_{\tilde f}C\frac{\Gamma(\gamma)}{|\varepsilon\omega_A -\lambda_0|^{\gamma}}(1+L_h)\tilde C \big(\|Bv_0\|_Y+ L_g(L_h+1)\|v_0\|_Y + \|v_0\|_{Y_1}\big)\\
    &\quad  + C_A L_{\tilde f} \frac{L_\zeta L_g C_B \Gamma(\delta)}{ (\eta -\zeta^{-1}( \varepsilon\omega_A+\lambda_0)- N_F^\zeta)^\delta} \frac{\Gamma(\gamma)}{|\varepsilon\omega_A -\lambda_0|^{\gamma}} \|v_0\|_{Y_1}\\
    &\quad +L_{\tilde f} \lambda_0^{-1}\varepsilon^{1-\alpha} C (L_h+1)M_B \textnormal{e}^{C_BL_g(1+L_h)\Gamma(\delta)} \|v_0\|_{Y_1}.
\end{align*}
 Focusing on the terms in $\varepsilon$ and using $\eta= \zeta^{-1}(\varepsilon\omega_A+\lambda_0) +\frac{ N_F^\zeta+ N_S^\zeta}{2}$ yields the desired estimate
 \begin{align*}
      \|h^{\varepsilon,\zeta}_{X_1}(\bar v_S(t))- h^0(\bar v_S(t))\|_{X_\alpha}  &\leq C\bigg( \varepsilon^{1-\alpha} +\frac{1}{(N_S^\zeta-N_F^\zeta)^\delta}\bigg) \|v_0\|_{Y_1}.
 \end{align*}
 This shows statement (GF2) of the generalized Fenichel Theorem \ref{general fenichel thm}.
\end{proof}

\begin{remark}
    If we have additional knowledge on the convergence of $N_S^\zeta-N_F^\zeta$ we can improve the above result.
    Assuming that we have the  asymptotic $N_S^\zeta-N_F^\zeta= C \zeta^{-\beta}$, for some $\beta>0$ then we obtain for $\varepsilon>0$ small enough
    \begin{align*}
        \|h^{\varepsilon,\zeta}_{X_1}(\bar v_S(t))- h^0(\bar v_S(t))\|_{X_1} &\leq C\bigg(\min\{ \varepsilon^{1-\gamma},\varepsilon^{\alpha-1}\} +\zeta^\beta \bigg) \|v_0\|_{Y_\alpha},
    \end{align*}
    where we used the condition $\varepsilon\zeta^{-1}<1$;
    for more background on the importance of double limits in singular perturbations we also refer to~\cite{KUEHNdouble}.
    \end{remark}

\subsection{Attraction of trajectories and approximation by the slow flow}

One technical result is to show the differentiability of the slow manifold.
Here, we use the assumption that the nonlinearities $f$ and $g$ are continuously differentiable such that
\begin{align*}
    \|\textnormal{D} f(x,y)\|_{\mathcal{B}(X_1\times Y_1,X_\gamma)}\leq L_f,\qquad \|\textnormal{D} g(x,y)\|_{\mathcal{B}(X_1\times Y_1,Y_\delta)}\leq L_g.
\end{align*}
\begin{prop}
Under the previous assumptions in this section the slow manifold is differentiable, i.e. the mapping
\begin{align*}
    (Y_S^\zeta,\|\cdot\|_{Y_1})\to (X_1,\|\cdot\|_{X_1})\times (Y_F^\zeta,\|\cdot\|_{Y_1}),\, v_0\mapsto  (h_{X_1}^{\varepsilon,\zeta}(v_0),h_{Y_F^\zeta}^{\varepsilon,\zeta}(v_0))
\end{align*}
is differentiable.
\end{prop}

\begin{proof}
The proof can be found in \cite[Prop 5.4]{hummel2022slow} where the key idea is to show that the
derivative exists as the best local linear approximation of the graph of the manifold.
\end{proof}
This result together with Proposition \ref{Prop 4.6} yields the proof of statement (GF3) in Theorem \ref{general fenichel thm}.
Moreover, this will allow us to prove the attraction of trajectories and the approximation by the slow flow which is the goal of the reminder of the section.\\

Next, let $(h_{X_1}^{\varepsilon,\zeta}(v_0),h_{Y_F^\zeta}^{\varepsilon,\zeta}(v_0),v_0) \in S_{\varepsilon,\zeta}$.
Moreover, let $(u,v_F,v_S)$ be the strict solution of \eqref{full eq splitting} with initial data $(h_{X_1}^{\varepsilon,\zeta}(v_0),h_{Y_F^\zeta}^{\varepsilon,\zeta}(v_0),v_0)$, i.e. it holds that
\begin{align*}
  \partial_t u(t)= \frac{1}{\varepsilon}\tilde A_\varepsilon u(t) +\frac{1}{\varepsilon}\tilde{f} (u(t),v_F(t),v_S(t)) 
\end{align*}
and let $(u^\varepsilon,v_F^\varepsilon,v_S^\varepsilon)$ be the strict solution to \eqref{full eq splitting} with initial data $(u_0,v_{0,F},v_{0,S})$.
From the invariance and differentiability of the slow manifold it follows that $u(t)= h_{X_1}^{\varepsilon,\zeta}(v_S(t))$ and therefore
\begin{align*}
    \partial_t u(t)&= \partial_t  h_{X_1}^{\varepsilon,\zeta}(v_S(t))= \textnormal{D} h_{X_1}^{\varepsilon,\zeta}(v_S(t))[\partial_t v_S(t)]\\
    &=\textnormal{D} h_{X_1}^{\varepsilon,\zeta}(v_S(t))[B v_S(t)+ \pr_{Y_S^\zeta} g(u(t),v_F(t),v_S(t))].
\end{align*}
Combining both equations for $t=0$ yields that
\begin{align}\label{eq 4.4}
    h_{X_1}^{\varepsilon,\zeta} (v_0)= \varepsilon (\tilde A_\varepsilon)^{-1} \textnormal{D} h_{X_1}^{\varepsilon,\zeta}(v_0) \big[B v_0+ \pr_{Y_S^\zeta} g(h_{X_1}^{\varepsilon,\zeta}(v_0) ,h_{Y_F^\zeta}^{\varepsilon,\zeta}(v_0),v_0)\big] - (\tilde A_\varepsilon)^{-1} f( h_{X_1}^{\varepsilon,\zeta}(v_0), h_{Y_F^\zeta}^{\varepsilon,\zeta}(v_0),v_0).
\end{align}
Similarly we obtain
\begin{align}\label{eq 4.5}
    h_{Y_F^\zeta}^{\varepsilon,\zeta} (v_0)= B_{Y_F^\zeta}^{-1} \textnormal{D} h_{Y_F^\zeta}^{\varepsilon,\zeta}(v_0) \big[B v_0+ \pr_{Y_S^\zeta} g(h_{X_1}^{\varepsilon,\zeta}(v_0) ,h_{Y_F^\zeta}^{\varepsilon,\zeta}(v_0),v_0)\big] - B_{Y_F^\zeta}^{-1}\pr_{Y_F^\zeta} g( h_{X_1}^{\varepsilon,\zeta}(v_0), h_{Y_F^\zeta}^{\varepsilon,\zeta}(v_0),v_0).
\end{align}
We note that these two relations hold for arbitrary $v_0\in Y_S^\zeta$.
Therefore it also holds for $ v_0= v^\varepsilon_S(t)$ for fixed $t\geq 0$ and we have
\begin{align}\label{eq 4.6}
    \partial_t  h_{X_1}^{\varepsilon,\zeta}(v^\varepsilon_S(t)) &= \textnormal{D}   h_{X_1}^{\varepsilon,\zeta}(v^\varepsilon_S(t))\big[B v^\varepsilon_S(t)+ \pr_{Y_S^\zeta} g(u^\varepsilon(t),v^\varepsilon_F(t),v^\varepsilon_S(t))\big],\\ \label{eq 4.7}
     \partial_t  h_{Y_F^\zeta}^{\varepsilon,\zeta}(v^\varepsilon_S(t)) &= \textnormal{D}   h_{Y_F^\zeta}^{\varepsilon,\zeta}(v^\varepsilon_S(t))\big[B v^\varepsilon_S(t)+ \pr_{Y_S^\zeta} g(u^\varepsilon(t),v^\varepsilon_F(t),v^\varepsilon_S(t))\big].
\end{align}
Thus, we can express the solution as a flow on the slow manifold.

The next step in the generalization of the Fenichel theorem is to show that the fast components of the fast-reaction system approach the solutions on the slow manifold with an exponential rate in time.
\begin{prop}\label{Prop 4.9}
For $\varepsilon$ and $\zeta$ small enough satisfying $\varepsilon\zeta^{-1}<1$ there exist constants $C,\, c>0$ such that
\begin{align*}
    \bigg\|\begin{pmatrix}
    u^\varepsilon(t)-h_{X_1}^{\varepsilon,\zeta}(v^\varepsilon_S(t))\\ v^\varepsilon_F(t) -h_{Y_F^\zeta}^{\varepsilon,\zeta}(v^\varepsilon_S(t))
    \end{pmatrix} \bigg\|_{X_1\times Y_1} \leq C \textnormal{e}^{-ct} \bigg\|\begin{pmatrix}
    u_0- h_{X_1}^{\varepsilon,\zeta}(v_{0,S})\\ v_{0,F}-h_{X_1}^{\varepsilon,\zeta} (v_{0,S})
    \end{pmatrix}\bigg\|_{X_1\times Y_1}.
\end{align*}
\end{prop}
\begin{proof}
Starting with the first component we compute
\begin{align*}
    u^\varepsilon(t)-h_{X_1}^{\varepsilon,\zeta}(v_S^\varepsilon(t))&= \textnormal{e}^{\varepsilon^{-1}t \tilde A_\varepsilon}(u_0-h_{X_1}^{\varepsilon,\zeta}(v_{0,S}))+\textnormal{e}^{\varepsilon^{-1}t \tilde A_\varepsilon}h_{X_1}^{\varepsilon,\zeta}(v_{0,S}) - h_{X_1}^{\varepsilon,\zeta}(v_S^\varepsilon(t))\\
    &\quad+ \varepsilon^{-1} \int_0^t \textnormal{e}^{\varepsilon^{-1}(t-s) \tilde A_\varepsilon} \tilde f(u^\varepsilon(s),v^\varepsilon_F(s),v^\varepsilon_S(s))\, \textnormal{d}s\\
    &= \textnormal{e}^{\varepsilon^{-1}t \tilde A_\varepsilon}(u_0-h_{X_1}^{\varepsilon,\zeta}(v_{0,S})) - \int_0^t \partial_s \big(\textnormal{e}^{\varepsilon^{-1}(t-s) \tilde A_\varepsilon}h_{X_1}^{\varepsilon,\zeta}(v^\varepsilon_S(s))\big) \, \textnormal{d}s\\
    &\quad+ \varepsilon^{-1} \int_0^t \textnormal{e}^{\varepsilon^{-1}(t-s) \tilde A_\varepsilon} \tilde f(u^\varepsilon(s),v^\varepsilon_F(s),v^\varepsilon_S(s))\, \textnormal{d}s\\
    &=\textnormal{e}^{\varepsilon^{-1}t \tilde A_\varepsilon}(u_0-h_{X_1}^{\varepsilon,\zeta}(v_{0,S})) + \varepsilon^{-1}\int_0^t  \textnormal{e}^{\varepsilon^{-1}(t-s) \tilde A_\varepsilon}  \tilde A_\varepsilon h_{X_1}^{\varepsilon,\zeta}(v^\varepsilon_S(s)) \, \textnormal{d}s\\
    &\quad-\int_0^t \textnormal{e}^{\varepsilon^{-1}(t-s) \tilde A_\varepsilon} \partial_s h_{X_1}^{\varepsilon,\zeta}(v^\varepsilon_S(s)) \, \textnormal{d}s + \varepsilon^{-1} \int_0^t \textnormal{e}^{\varepsilon^{-1}(t-s) \tilde A_\varepsilon} \tilde f(u^\varepsilon(s),v^\varepsilon_F(s),v^\varepsilon_S(s))\, \textnormal{d}s.
    \intertext{Applying relations \eqref{eq 4.4} and \eqref{eq 4.6} yields}
    &=\textnormal{e}^{\varepsilon^{-1}t \tilde A_\varepsilon}(u_0-h_{X_1}^{\varepsilon,\zeta}(v_{0,S})) \\
    &\quad+ \int_0^t \textnormal{e}^{\varepsilon^{-1}(t-s) \tilde A_\varepsilon} \textnormal{D}h_{X_1}^{\varepsilon,\zeta}(v^\varepsilon_S(s))\pr_{Y_S^\zeta}\big[g(u^\varepsilon,v^\varepsilon_F,v^\varepsilon_S)-g(h_{X_1}^{\varepsilon,\zeta}(v^\varepsilon_S(s)),h_{Y_F^\zeta}^{\varepsilon,\zeta}(v^\varepsilon_S(s)),v^\varepsilon_S(s))\big]\, \textnormal{d}s\\
    &\quad + \varepsilon^{-1}\int_0^t \textnormal{e}^{\varepsilon^{-1}(t-s) \tilde A_\varepsilon} \big[f(u^\varepsilon,v^\varepsilon_F,v^\varepsilon_S)-f(h_{X_1}^{\varepsilon,\zeta}(v^\varepsilon_S(s)),h_{Y_F^\zeta}^{\varepsilon,\zeta}(v^\varepsilon_S(s)),v^\varepsilon_S(s))\big]\, \textnormal{d}s.
\end{align*}
Similarly, we obtain for the second component
\begin{align*}
    v^\varepsilon_F(t)-h_{Y_F^\zeta}^{\varepsilon,\zeta}(v_S^\varepsilon(t))&= \textnormal{e}^{t B}(v_{0,F}-h_{Y_F^\zeta}^{\varepsilon,\zeta}(v_{0,S})) \\
    &\quad+ \int_0^t \textnormal{e}^{(t-s) B} \textnormal{D}h_{Y_F^\zeta}^{\varepsilon,\zeta}(v^\varepsilon_S(s))\pr_{Y_S^\zeta}\big[g(u^\varepsilon,v^\varepsilon_F,v^\varepsilon_S)-g(h_{X_1}^{\varepsilon,\zeta}(v^\varepsilon_S(s)),h_{Y_F^\zeta}^{\varepsilon,\zeta}(v^\varepsilon_S(s)),v^\varepsilon_S(s))\big]\, \textnormal{d}s\\
    &\quad + \int_0^t \textnormal{e}^{(t-s) B} \pr_{Y_F^\zeta} \big[g(u^\varepsilon,v^\varepsilon_F,v^\varepsilon_S)-g(h_{X_1}^{\varepsilon,\zeta}(v^\varepsilon_S(s)),h_{Y_F^\zeta}^{\varepsilon,\zeta}(v^\varepsilon_S(s)),v^\varepsilon_S(s))\big]\, \textnormal{d}s.
\end{align*}
Setting $z(t)= \| u^\varepsilon(t)-h_{X_1}^{\varepsilon,\zeta}(v_S^\varepsilon(t))\|_{X_1}+ \| v^\varepsilon_F(t)-h_{Y_F^\zeta}^{\varepsilon,\zeta}(v_S^\varepsilon(t))\|_{Y_1}$
we can estimate the two components as follows
\begin{align*}
    z(t)&\leq M_A \textnormal{e}^{\varepsilon^{-1}(\varepsilon\omega_A+\lambda_0)}  \| u^\varepsilon(t)-h_{X_1}^{\varepsilon,\zeta}(v_S^\varepsilon(t))\|_{X_1} + M_B  \textnormal{e}^{(\zeta^{-1}(\varepsilon\omega_A+\lambda_0) +N_F^\zeta)t} \| v^\varepsilon_F(t)-h_{Y_F^\zeta}^{\varepsilon,\zeta}(v_S^\varepsilon(t))\|_{Y_1} \\
    &\quad + C_A L_f(\varepsilon L_g+1) \int_0^t \frac{\textnormal{e}^{\varepsilon^{-1}(\varepsilon\omega_A -\lambda_0)(t-s)}}{\varepsilon^\gamma (t-s)^{1-\gamma}} z(s)\, \textnormal{d}s\\
    &\quad + C_B L_g( L_g+1) \int_0^t \frac{\textnormal{e}^{(\zeta^{-1}(\varepsilon\omega_A+\lambda_0) +N_F^\zeta )(t-s)}}{(t-s)^{1-\delta}} z(s)\, \textnormal{d}s.
\end{align*}
Applying a version of Gronwall's inequality yields
\begin{align*}
    z(t)\leq C \textnormal{e}^{-ct}\bigg\|\begin{pmatrix}
    u_0- h_{X_1}^{\varepsilon,\zeta}(v_{0,S})\\ v_{0,F}-h_{X_1}^{\varepsilon,\zeta} (v_{0,S})
    \end{pmatrix}\bigg\|_{X_1\times Y_1}.
\end{align*}
This completes the proof.
\end{proof}

The last step in the generalization of Fenichel's theorem is to show that the semi-slow flow is a good approximation of the semi-flow of the original fast-reaction system.

So far we have measured the distance of the slow manifold to the subset $S_{0,\zeta}$ of the critical manifold given by
\begin{align*}
    S_{0,\zeta}:=\{ (h^0(v_0),v_0)\in S_0\,:\, \pr_{Y_F^\zeta}v_0=0\}.
\end{align*}
However, this subset is in general not invariant under the slow flow. 
Hence, we introduce the reduced slow subsystem
\begin{align}
\label{eq. reduced slow}
\begin{split}
    0&= f(h^0(v^0_\zeta(t)),v^0_\zeta(t)),\\
    0&= \pr_{Y_F^\zeta} v^0_\zeta(t),\\
    \partial_t v^0_\zeta(t)&= B v^0_\zeta(t)+ \pr_{Y_S^\zeta} g(h^0(v^0_\zeta(t)),v^0_\zeta(t)),\\
    v^0_\zeta(0)&= \pr_{Y_S^\zeta}v_0,
\end{split}
\end{align}
 and we observe that $S_{0,\zeta}$ is forward invariant under the flow generated by the slow subsystem.

\begin{prop}
For all $T>0$, there exists a constant $C>0$ such that for all $t\in [0,T]$ and all $\zeta>0$ small enough the following holds
\begin{align*}
    \|v^0(t)-v^0_\zeta(t)\|_{Y_1} \leq  C\bigg(\|\pr_{Y_F^\zeta} v_0\|_{Y_1}+ \frac{\|v_0\|_{Y_1}}{(\omega_B -\zeta^{-1}\lambda_0 -N_F^\zeta)^\delta}\bigg).
\end{align*}
% C=\textnormal{e}^{\omega_B t} \textnormal{e}^{\textnormal{e}^{\omega_B t}}
\end{prop}
\begin{proof}
The well-posedness of the slow subsystem follows from Prop \ref{Prop 3.2} and thus we can write the solution as
\begin{align*}
    v^0_\zeta(t)= \textnormal{e}^{tB} \pr_{Y_S^\zeta} v_0+ \int_0^t \textnormal{e}^{(t-s)B} \pr_{Y_S^\zeta} g(h^0(v^0_\zeta(s)),v^0_\zeta(s))\, \textnormal{d}s.
\end{align*}
Hence
\begin{align*}
    v^0(t)-v^0_\zeta(t)=& \textnormal{e}^{tB} \pr_{Y_F^\zeta} v_0+ \int_0^t \textnormal{e}^{(t-s)B} \pr_{Y_S^\zeta}\big( g(h^0(v^0(s)),v^0(s))-g(h^0(v^0_\zeta(s)),v^0_\zeta(s))\big)\, \textnormal{d}s\\
    &- \int \textnormal{e}^{(t-s)B}\pr_{Y_F^\zeta} g(h^0(v^0_\zeta(s)),v^0_\zeta(s))\, \textnormal{d}s.
\end{align*}
Now, we can estimate
\begin{align*}
    \|v^0(t)-v^0_\zeta(t)\|_{Y_1} &\leq M_B \textnormal{e}^{(\zeta^{-1}\lambda_0 +N_F^\zeta)t} \|\pr_{Y_F^\zeta} v_0\|_{Y_1}\\
    &\quad + L_gC_B\int_0^t \frac{\textnormal{e}^{(\zeta^{-1}\lambda_0 +N_F^\zeta)(t-s)}}{(t-s)^{1-\delta}}\big( \|h^0(v^0(s))\|_{X_1}+ \|v^0(s)\|_{Y_1}\big)\, \textnormal{d}s\\
    &\quad+ L_GC_B \int_0^t \frac{\textnormal{e}^{\omega_B(t-s)}}{(t-s)^{1-\delta}} \big( \|h^0(v^0(s))-h^0(v^0_\zeta(s))\|_{X_1}+ \|v^0(s)-v^0_\zeta(s)\|_{Y_1}\big)\, \textnormal{d}s\\
    &\leq M_B \textnormal{e}^{(\zeta^{-1}\lambda_0+N_F^\zeta)t} \|\pr_{Y_F^\zeta} v_0\|_{Y_1}\\
    &\quad+  L_g(L_h+1) M_B C_B \textnormal{e}^{C_BL_g(1+L_h)\Gamma(\delta)} \textnormal{e}^{\omega_B t} \|v_0\|_{Y_1} \int_0^t \frac{\textnormal{e}^{(\zeta^{-1}\lambda_0 +N_F^\zeta-\omega_B)(t-s)}}{(t-s)^{1-\delta}}  \, \textnormal{d}s\\
    &\quad+ L_G (L_h+1) C_B \int_0^t \frac{\textnormal{e}^{\omega_B(t-s)}}{(t-s)^{1-\delta}}  \|v^0(s)-v^0_\zeta(s)\|_{Y_1}\, \textnormal{d}s\\
    &\leq C \textnormal{e}^{\omega_B t}\bigg(\|\pr_{Y_F^\zeta} v_0\|_{Y_1}+ \frac{\|v_0\|_{Y_1}}{(\omega_B -\zeta^{-1}\lambda_0 -N_F^\zeta)^\delta}\bigg)\\
    &\quad+L_G (L_h+1) C_B \int_0^t \frac{\textnormal{e}^{\omega_B(t-s)}}{(t-s)^{1-\delta}}  \|v^0(s)-v^0_\zeta(s)\|_{Y_1}\, \textnormal{d}s.
\end{align*}
Applying Gronwall's inequality yields the desired result.
\end{proof}

\begin{remark}
    In the case that the operator $B$ is dissipative, i.e. $\| \textnormal{e}^{tB}\|_{\mathcal{B}(Y_1)}\leq \textnormal{e}^{\omega_B t}$ with $\omega_B<0$ the above result reduces to 
    \begin{align*}
        \|v^0(t)-v^0_\zeta(t)\|_{Y_1} \leq  C \textnormal{e}^{\omega_B t} \bigg(\|\pr_{Y_F^\zeta} v_0\|_{Y_1}+ \frac{\|v_0\|_{Y_1}}{|\omega_B -\zeta^{-1}\lambda_0 -N_F^\zeta|^\delta}\bigg).
    \end{align*}
\end{remark}

\begin{prop}
For all $T>0$ there exists a constant $C>0$ such that for all $t\in [0,T]$ and all $\zeta>0$ small enough it holds that
\begin{align*}
    \bigg\| \begin{pmatrix}
    u^\varepsilon(t)- h^0(v^0_\zeta(t)\\ v^\varepsilon(t)- v_\zeta^0(t)
    \end{pmatrix}\bigg \|_{X_\alpha \times Y_1} &\leq C  {\textnormal{e}^{\textnormal{e}^{\omega_{\small{B}} t}+ct}} t^\delta   \varepsilon^{1-\alpha}  (\textnormal{e}^{\omega_B t}  + \varepsilon^{-\gamma}t^\gamma\textnormal{e}^{\varepsilon^{-1}(\varepsilon\omega_A+\lambda_0) t}) \|v_0\|_{Y_1}\\
     &\quad +C{\textnormal{e}^{\textnormal{e}^{\omega_{\small{B}} t}+ct}} C M_A\textnormal{e}^{\varepsilon^{-1}(\varepsilon\omega_A+\lambda_0) t}\|u_0- h^0(v_0)\|_{X_1}\\
     &\quad+C \textnormal{e}^{\omega_B t} \textnormal{e}^{\textnormal{e}^{\omega_B t}}\bigg(\|\pr_{Y_F^\zeta} v_0\|_{Y_1}+ \frac{\|v_0\|_{Y_1}}{(\omega_B -\zeta^{-1}\lambda_0 -N_F^\zeta)^\delta}\bigg).
\end{align*}
If in addition the initial values are on the slow manifold it holds that
\begin{align*}
    &\bigg\| \begin{pmatrix}
    u^\varepsilon(t)- h^0(v^0_\zeta(t)\\ v^\varepsilon(t)- v_\zeta^0(t)
    \end{pmatrix}\bigg \|_{X_\alpha \times Y_1}\leq \\
    &\quad\leq C \textnormal{e}^{(\omega_B+c)t}  \|v_0\|_{Y_1}\bigg(\varepsilon^{1-\alpha} t^\delta + \varepsilon^{-\gamma}t^\gamma\textnormal{e}^{\varepsilon^{-1}(\varepsilon\omega_A+\lambda_0)t}+\frac{1}{(\omega_B -\zeta^{-1}\lambda_0 -N_F^\zeta)^\delta}+ \frac{1}{(N_S^\zeta-N_F^\zeta)^\delta}\bigg).
\end{align*} 

\end{prop}
\begin{proof}
The first estimate is a consequence of Theorem \ref{Thm 3.6} and Proposition \ref{Prop 4.9}. 
For the second estimate we apply the triangle inequality to the terms
\begin{align*}
    \|\pr_{Y_F^\zeta} v_0\|_{Y_1}&\leq \|pr_{Y_F^\zeta} v_0-h_{Y_F^\zeta}^{\varepsilon,\zeta}(v_0)\|_{Y_1}+ \|h_{Y_F^\zeta}^{\varepsilon,\zeta}(v_0)\|_{Y_1},\\
    \|u_0-h^0(v_0)\|_{X_1}&\leq \|u_0-h^{\varepsilon,\zeta}_{X_1}(v_0)\|_{X_1}+ \|h^{\varepsilon,\zeta}_{X_1}(v_0)-h^0(v_0)\|_{X_1}.
\end{align*}
The first term in each of the two estimates equals to zero since the initial data is on the slow manifold and the second term can be estimated by Proposition \ref{Prop 4.6}.
This then proofs the fourth and last statement (GF4) of the generalized Fenichel Theorem \ref{general fenichel thm}.
\end{proof}

\begin{remark}
Again we want to put the final result into context. 
To this end we assume that the constant $N_F^\zeta$ is bounded independent of $\zeta$ and that in the asymptotic limit $N_S^\zeta-N_F^\zeta= C \zeta^{-\beta}$.
Then, using the assumption that $\varepsilon \zeta^{-1}<1$ we obtain
\begin{align*}
   &\bigg\| \begin{pmatrix}
    u^\varepsilon(t)- h^0(v^0_\zeta(t)\\ v^\varepsilon(t)- v_\zeta^0(t)
    \end{pmatrix}\bigg \|_{X_\alpha \times Y_1}\leq 
    C \textnormal{e}^{(\omega_B+c)t}  \|v_0\|_{Y_1}\big(\varepsilon^{\alpha-1}  + \varepsilon^{-\gamma}\textnormal{e}^{\varepsilon^{-1}(\varepsilon\omega_A+\lambda_0)t}+\zeta^{\beta} +\zeta^{\beta\delta} \big).  
\end{align*}
Letting $\varepsilon,\,\zeta\to 0$ yields an explicit rate for the convergence of the semi-flow fast-reaction system to the semi-flow of the slow subsystem.

\end{remark}

\subsection{Example}\label{subsection example}

As a final step we apply our method to a model from biology, the Beddington-DeAngelis type predator-prey model \cite{cantrell2001dynamics,conforto2018reaction}.
Let $\Omega\subset \mathbb{R}^n$, then the system reads
\begin{align*}
    \partial_t p_h^\varepsilon &= d_h \Delta p_h^\varepsilon+\frac{1}{\varepsilon}\big(a N^\varepsilon p_s^\varepsilon -\tilde \gamma p_h^\varepsilon\big) -\mu p_h^\varepsilon,\\
    \partial_t p_s^\varepsilon &= d_s \Delta p_s^\varepsilon-\frac{1}{\varepsilon}\big(a N^\varepsilon p_s^\varepsilon -\tilde \gamma p_h^\varepsilon\big )-\mu p_s^\varepsilon +\Gamma p^\varepsilon_h,\\
    \partial_t N^\varepsilon &= d_N \Delta N^\varepsilon +r_0(1-\eta N^\varepsilon)N^\varepsilon- a N^\varepsilon p_s^\varepsilon,
\end{align*}
subject to homogeneous Neumann boundary conditions.

To bring this system in the form of \eqref{full eq} we introduce a new variable $P^\varepsilon= p_h^\varepsilon+p_s^\varepsilon$ and add the first and second equation to obtain the evolution of $P^\varepsilon$. Moreover, we assume that the diffusion coefficients satisfy $d_h=d_s$. This ensures the semi-linear structure of the equations and we have
\begin{align}\label{Beddington system}
    \begin{split}
    \partial_t p_h^\varepsilon &= d \Delta p_h^\varepsilon+\frac{1}{\varepsilon}\big(a N^\varepsilon (P^\varepsilon-p_h^\varepsilon) -\tilde \gamma p_h^\varepsilon\big) -\mu p_h^\varepsilon,\\
    \partial_t P^\varepsilon &= d \Delta P^\varepsilon -\mu P^\varepsilon +\Gamma p^\varepsilon_h,\\
    \partial_t N^\varepsilon &= d_N \Delta N^\varepsilon +r_0(1-\eta N^\varepsilon)N^\varepsilon- a N^\varepsilon(P^\varepsilon-p_h^\varepsilon).
    \end{split}
\end{align}
The limiting system as $\varepsilon\to 0$ has the form
\begin{align}\label{Beddington limit}
    \begin{split}
        0 &=a N^0 (P^0-p_h^0) -\tilde \gamma p_h^0,\\
    \partial_t P^0&= d \Delta P^0-\mu P^0 +\Gamma p^0_h,\\
    \partial_t N^0 &= d_N \Delta N^0 +r_0(1-\eta N^0)N^0- a N^0(P^0-p_h^0).
    \end{split}
\end{align}
The case where $d_h\neq d_s$ is studied in an upcoming work.
We note that in \cite{conforto2018reaction,soresina2023fast} the convergence of solutions of \eqref{Beddington system} to \eqref{Beddington limit} in some $L^p$-spaces was shown using entropy methods. However, we seek the convergence in more regular spaces and also aim to construct the slow manifold in these more regular spaces. 

Moreover, as we require global Lipschitz continuity in the nonlinear terms we modify the above system. Let $\nu>0$ be a small parameter. 
Then, the modified system reads
\begin{align}\label{modified Beddington system}
    \begin{split}
    \partial_t p_h^\varepsilon &= (d \Delta-\mu\textnormal{Id}) p_h^\varepsilon+\frac{1}{\varepsilon}\bigg(a \frac{N^\varepsilon (P^\varepsilon-p_h^\varepsilon)}{1+\nu((N^\varepsilon)^2+(P^\varepsilon)^2)} -\tilde \gamma p_h^\varepsilon\bigg),\\
   \partial_t P^\varepsilon &= d \Delta P^\varepsilon -\mu P^\varepsilon +\Gamma p^\varepsilon_h,\\
    \partial_t N^\varepsilon &= d_N \Delta N^\varepsilon +r_0(1-\eta N^\varepsilon)N^\varepsilon- a \frac{N^\varepsilon (P^\varepsilon-p_h^\varepsilon)}{1+\nu((N^\varepsilon)^2+(P^\varepsilon)^2)}.
    \end{split}
\end{align}
Here, the new limiting system as $\varepsilon\to 0$ has the form
\begin{align}\label{modified Beddington limit}
    \begin{split}
        0 &=a \frac{N^0(P^0-p_h^0)}{1+\nu( (N^0)^2+(P^0)^2)}-\tilde \gamma p_h,\\
    \partial_t P^0&= d \Delta P^0-\mu P^0 +\Gamma p_h^0,\\
    \partial_t N^0 &= d_N \Delta N^0 +r_0(1-\eta N^0)N^0- a \frac{N^0(P^0-p_h^0)}{1+\nu( (N^0)^2+(P^0)^2)}.
    \end{split}
\end{align}
Now, let $\Omega=\mathbb{T}^n=[0,2\pi]^n$ with $n=1,2,3$ and let $X=Y=L^2(\Omega)$ be the underlying Banach space with $X_1=Y_1=H^2(\Omega)$.
Moreover, let the initial data $N_0,P_0,p_{h,0}\in H^2(\Omega)$ satisfy the critical manifold constraint
\begin{align} \label{initial data constraint}
  -\tilde \gamma p_{h,0} +a \frac{N_0 (P_0-p_{h,0})}{1+\nu( (N_0)^2+(P_{0})^2)}=0.  
\end{align}

\begin{theorem}
Let $(p_h^\varepsilon,P^\varepsilon,N^\varepsilon)$ be the solution of \eqref{Beddington system} and denote $(p_h^0,P^0,N^0)$ the solution of the limit system \eqref{Beddington limit}.
Then there exist constants $C,\, c>0$ such that for all $\varepsilon\in (0,\varepsilon_0]$ and all $N_0,P_0,p_{h,0} \in H^2(\Omega)$ satisfying \eqref{initial data constraint} with $\|N_0\|_{H^2(\Omega)},\|P_0\|_{H^2(\Omega)},\|p_{h,0}\|_{H^2(\Omega)}\leq \sigma$ it holds
\begin{align*}
    \sup_{0\leq t\leq T(\sigma)}\big(  \|p_h^\varepsilon(t)- p_h^0(t)\|_{H^{2\alpha}(\Omega)} +  \|P^\varepsilon(t)- P^0(t)\|_{H^2(\Omega)}+ \|N^\varepsilon(t)- N^0(t)\|_{H^2(\Omega)} \big) \leq C\big(\varepsilon^{1-\alpha} +\textnormal{e}^{-\varepsilon^{-1} c t}\big),
\end{align*}
where $\alpha\in [0,1)$ such that $2\alpha>\frac{n}{2}$ and the final time $T(\sigma)$ is given by
\begin{align*}
    T(\sigma):= \inf\big\{ t\in [0,\infty):\max\{ &\|p_h^\varepsilon(t)\|_{H^2(\Omega)}, \|P^\varepsilon(t)\|_{H^2(\Omega)}, \|N^\varepsilon(t)\|_{H^2(\Omega)},\\
     & \|p_h^0(t)\|_{H^2(\Omega)}, \|P^0(t)\|_{H^2(\Omega)}, \|N^0 (t)\|_{H^2(\Omega)}\}>\sigma\big\}.
\end{align*}
\end{theorem}

\begin{theorem}
Let $n=1$ and consider the domain $\Omega=\mathbb{T}=[0,2\pi]$ and let the assumptions of the previous result hold. 
Moreover, there is a $\zeta_0$ such that for all $\zeta\in (0,\zeta_0]$ satisfying $\varepsilon\zeta^{-1}<\frac{1}{2}$ the following holds:
\begin{itemize}
    \item There exists a slow manifold $S_{\varepsilon,\zeta}$ that is $C^1$-regular and locally invariant under the semi-flow generated by \eqref{Beddington system}.
    \item Let $S_{0,\zeta}:=\{ (p_h,P,N)\in S_0\, :\, (P,N)\in B_{Y_S^\zeta}(0,\sigma)\}$ be the submanifold of the critical manifold which consists of all points whose slow components are in the bounded set $ B_{Y_S^\zeta}(0,\sigma)$.
    Then there exists a constant $C>0$ such that
    \begin{align*}
        \dist (S_{\varepsilon,\zeta},S_{0,\zeta})\leq C\big(\varepsilon^{1-\alpha}+ \zeta^{1/2}\big).
    \end{align*}
    \item Suppose that $(h^0(P_0,N_0),P_0,N_0)\in S_{\varepsilon,\zeta}$ satisfying $\|P_0\|_{H^2([0,2\pi])},\|N_0\|_{H^2([0,2\pi])} ,\|h^0(P_0,N_0)\|_{H^2([0,2\pi])}\leq \sigma$.
    Then there exists a constant $C>0$ such that 
    \begin{align*}
       \sup_{0\leq t\leq T(\sigma)}  \big(  \|p_h^\varepsilon(t)- p_h^0(t)\|_{H^{2\alpha}(\Omega)} +  \|P^\varepsilon(t)- P^0(t)\|_{H^2(\Omega)}+ \|N^\varepsilon(t)- N^0(t)\|_{H^2(\Omega)} \big) \leq C\big(\zeta^{1-\alpha} +\zeta^{1/2}\big).
    \end{align*}
\end{itemize}
\end{theorem}

What remains is to check the assumptions from sections \ref{sect 3.1} and \ref{sect 4.1}.
\begin{itemize}
    \item[(i)] The Laplacian generates an analytic $C_0$-semigroup $(\textnormal{e}^{t\Delta})_{t\geq 0}$ on $X_1=H^2(\Omega)\subset L^2(\Omega)$ and from complex interpolation it follows that the spaces $X_\alpha= H^{2\alpha}(\Omega)$ are a suitable choice for the Banach scales.
    By restricting $\alpha\in (\frac{n}{4},1)$ it follows that $H^{2\alpha}(\Omega)$ and $H^2(\Omega)$ are Banach algebras.

    \item[(ii)] In this example we have $A=B_1= d \Delta -\mu \textnormal{Id}$, which implies that the growth bound satisfies $\omega_A=\omega_{B_1}\leq -\mu$. Similarly, we have $B_2=d_N \Delta$ and $\omega_{B_2}\leq 0$.
    
    \item[(iii)] The nonlinearities of the system are 
    \begin{align*}
        f(p_h,P,N)&= a \frac{N (P-p_h)}{1+\nu( N^2+P^2)}-\tilde\gamma p_h ,\\
        g_1(p_h,P,N)&=\Gamma p_h,\quad\textnormal{and}\\
        g_2(p_h,P,N)&= r_0(1-\eta N)N -a \frac{N (P-p_h)}{1+\nu(N^2+P^2)}.
    \end{align*}
    As the nonlinearities do not decrease the regularity we set $\gamma=\delta=1$.
    In addition, the function $f$ is Fr\'echet differentiable as a mapping $f: H^{2\alpha}\times H^{2\alpha}\times H^{2\alpha}\to H^{2\alpha}$ and similar the functions $f,g_1,g_2$ are Fr\'echet differentiable as mappings $f,g_1,g_2: H^2\times H^2\times H^2\to H^{2}$.
    
    As $0\leq p_h\leq P$ by the definition of $P$ it follows that there exists a constant $L_f>0$ such that
    \begin{align*}
        \|\textnormal{D} f(x,y,z)\|_{\mathcal{B}(H^2\times H^2\times H^2,H^{2})}&\leq L_f, \quad\textnormal{and}\quad  \|\textnormal{D} f(x,y,z)\|_{\mathcal{B}(H^{2\alpha}\times H^{2\alpha}\times H^2,H^{2\alpha})}\leq L_f.
    \end{align*}
    This implies that the following inequalities hold globally
        \begin{align*}
             \|f(x_1,y_1,z_1)-f(x_2,y_2,z_1)\|_{H^{2}} &\leq L_f\big( \|x_1-x_2\|_{H^{2}}+\|y_1-y_2\|_{H^{2}} +\|z_1-z_2\|_{H^{2}} \big),\\
             \|f(\tilde x_1,\tilde y_1,\tilde z_1)-f(\tilde x_2,\tilde y_2,\tilde z_1)\|_{H^{2\alpha}}&\leq L_f\big( \|\tilde x_1-\tilde x_2\|_{H^{2\alpha}}+\|\tilde y_1-\tilde y_2\|_{H^{2\alpha}} +\|\tilde z_1-\tilde z_2\|_{H^{2\alpha}}\big),
        \end{align*}
        where $x_1,x_2,y_1,y_2,z_1,z_2\in H^2(\Omega)$ and $\tilde x_1,\tilde x_2,\tilde y_1,\tilde y_2,\tilde z_1,\tilde z_2\in H^{2\alpha}$ and where the Lipschitz constant is bounded by $L_f\leq \max\{a*\nu^{-1},\tilde \gamma\}$.
        
    Although in \cite{conforto2018reaction}and \cite{soresina2023fast} it was shown that $N$ has a global $L^\infty$-bound $N_\infty$, this is not enough to guarantee the global Lipschitz continuity for $g_2$ in the modified systems \eqref{modified Beddington system}.
    To ensure this we apply a cutoff technique. 
    Let $\sigma>0$ and choose a $C^\infty$-cutoff function $\chi : H^2(\Omega)\to [0,1]$ such that $\chi(x)=1$ if $x\in B(0,\sigma)$, $\chi(x)=0 $ if $x\in H^2(\Omega)\setminus B(0,2\sigma)$ and $\|\textnormal{D}\chi\|_{\mathcal{B}(H^2(\Omega),\mathbb{R})}\leq \sigma$.
    Then we have that
    \begin{align*}
        g_{2,\chi}: H^2\times H^2\times H^2&\to H^{2},\\
        (x,y,z)&\mapsto r_0\big(1-\eta \chi(\|z\|_{H^2})z\big)\chi(\|z\|_{H^2})z -a \frac{z(y-x)}{1+\nu(z^2+y^2)}
    \end{align*}
    satisfies
    \begin{align*}
        \|\textnormal{D} g_{2,\chi}(x,y,z)\|_{\mathcal{B}(H^2\times H^2\times H^2,H^{2})}&\leq L_{g_{2}}.
    \end{align*}
    which implies 
    \begin{align*}
         \|g_{2,\chi}(x_1,y_1,z_1)-g_{2,\chi}(x_2,y_2,z_1)\|_{H^{2}} &\leq L_{g_2}\big( \|x_1-x_2\|_{H^{2}}+\|y_1-y_2\|_{H^{2}} +\|z_1-z_2\|_{H^{2}} \big)
    \end{align*}
    for $x_1,x_2,y_1,y_2,z_1,z_2\in H^2(\Omega)$.
    \item[(iv)] We compute that $\textnormal{D}_{p_h}f= -a N -\tilde \gamma\leq -\tilde \gamma <0$ as solutions to \eqref{modified Beddington limit} satisfy $N\geq0$. Hence, the nonlinearity $f$ satisfies the assumptions of the implicit function theorem and we obtain the existence of a function $h^0:H^2\times H^2\to H^2$ such that $h^0(P^0,N^0)=p_h^0$ for all $(P^0,N^0)\in V\subset H^2\times H^2$.
  \item[(v)] The function $h^0$ given by the implicit function theorem can be computed explicitly, i.e. 
  \begin{align*}
      h^0(P^0,N^0)= \frac{a N^0 P^0}{\tilde \gamma+ aN^0+\tilde \gamma \nu ((N^0)^2+(P^0)^2)} .
  \end{align*}
  From this explicit form we immediately see that $h^0:H^2\times H^2\to H^2$ is Lipschitz continuous with constant $L_{h^0}$.
\end{itemize}
Thus the limit system \eqref{modified Beddington limit} can be written as
\begin{align}\label{modified beddington limit slow}
    \begin{split}
        \partial_t P^0&= d \Delta P^0-\mu P^0 +\Gamma\frac{a N^0 P^0}{\tilde \gamma+ aN^0+\tilde \gamma \nu ((N^0)^2+(P^0)^2)},\\
    \partial_t N^0 &= d_N \Delta N^0 +r_0(1-\eta N^0)N^0- \frac{aN\bigg(P-\frac{a N^0 P^0}{\tilde \gamma+ aN^0+\tilde \gamma \nu ((N^0)^2+(P^0)^2)}\bigg)}{1+\nu((N^0)^2+(P^0)^2)},
    \end{split}
\end{align}
defined on the slow manifold
\begin{align}\label{beddingto slow manifold}
    S_0=\bigg\{ (p_h,P,N)\in (H^2)^3: p_h=\frac{a N^0 P^0}{\tilde \gamma+ aN^0+\tilde \gamma \nu ((N^0)^2+(P^0)^2)} \}.
\end{align}
\begin{itemize}
    \item[(vi)] Using that $\textnormal{D}_{p_h}f(p_h,P,N)\leq -\tilde \gamma$, we can introduce the operator $\tilde A_\varepsilon$ as 
    \begin{align*}
        \tilde A_\varepsilon=\varepsilon d \Delta- (\varepsilon\mu+\tilde \gamma) \textnormal{Id}
    \end{align*}
    and the corresponding nonlinear function 
    \begin{align*}
        \tilde f= a \frac{N(P-p_h)}{1+\nu(N^2+P^2)}.
    \end{align*}
    It follows that $\tilde A_\varepsilon$ generates an analytic semigroup with growth bound $\varepsilon(\omega_A-\mu)-\tilde \gamma<0$.
    Moreover, we have that the Lipschitz constant of $\tilde f$ is bounded by $L_{\tilde f}\leq a\nu^{-1}$.
\end{itemize}
Thus, all assumptions in section \ref{sect 3.1} are satisfied and it remains to show that the Laplacian satisfies the assumptions for the splitting.

Making use of the Hilbert structure of $L^2([0,2\pi])$ we introduce the splitting as 
\begin{align*}
    L^2=Y= Y_S^\zeta\oplus Y_F^\zeta
\end{align*} 
by truncating at a certain Fourier mode.
Assuming that 
$$-(k_0+1)^2< \zeta^{-1}(\varepsilon (\omega_A-\mu) -\tilde \gamma)\leq - k_0^2$$
we set 
\begin{align*}
  Y_S^\zeta= \text{span}\{x\mapsto \textnormal{e}^{ikx}: |k|\leq k_0\},\qquad Y_F^\zeta= \text{cl}_{L^2}(\text{span} \{x\mapsto \textnormal{e}^{ikx}: |k|> k_0\}). 
\end{align*}

\begin{itemize}
    \item[(vii)] Since $Y_S^\zeta$ is a finite-dimensional space the shifted Laplacian $\Delta- \mu \textnormal{Id}$ generates a $C_0$-group on $Y_S^\zeta$, where $(\textnormal{e}^{t(\Delta- \mu \textnormal{Id})_{Y_S^\zeta}})_{t\in \mathbb{R}}$ coincides with $(\textnormal{e}^{t(\Delta- \mu \textnormal{Id})})|_{Y_S^\zeta}$ for $t\geq 0$.
    \item[(viii)] For the realization of the Laplacian in $Y_F^\zeta$ we have that $0$ is in its resolvent.
    This can be seen by the representation of the inverse, which is given by
    \begin{align*}
        \Delta_{Y_F^\zeta}^{-1}: L^2([0,2\pi]) \to  H^2([0,2\pi]),\quad \sum_{|k|\geq k_0} \hat f(k) \textnormal{e}^{ikx} \mapsto \sum_{|k|\geq k_0} \frac{\hat f(k) \textnormal{e}^{ikx}}{|k|^2}
    \end{align*}
    and is well-defined as $k=0$ does not appear in the sum.
    \item[(ix)] We note that the semigroup $(\textnormal{e}^{t\Delta})_{t\geq 0}$ can be expressed as
    \begin{align*}
         \textnormal{e}^{t(\Delta-\mu \textnormal{Id})} f= \bigg[ x\mapsto \sum_{k\in \mathbb{Z}} \textnormal{e}^{-(|k|^2+\mu) t} \hat f (k) \textnormal{e}^{ikx}\bigg].
    \end{align*}
    Applying Plancherel's theorem for $y_S\in Y_S^\zeta$ and $t\geq 0$ yields
    \begin{align*}
        \|\textnormal{e}^{-t(\Delta-\mu \textnormal{Id})} y_S\|_{H^2([0,2\pi])}\leq \textnormal{e}^{(k_0-1)^2t}\|y_S\|_{H^2([0,2\pi])}.
    \end{align*}
    Thus, we can take $N_S^\zeta= \zeta^{-1}\tilde \gamma -(k_0-1)^2$ and similarly we may take $N_F^\zeta= \zeta^{-1} \tilde \gamma- k_0^2$.
    Then,
    \begin{align*}
        N_S^\zeta-N_F^\zeta= 2 k_0-1 \geq 2 \sqrt{\zeta^{-1}\tilde \gamma}-C.
    \end{align*}
    Therefore, $N_S^\zeta-N_F^\zeta\approx \mathcal{O}(\zeta^{-1/2})$ for $\zeta$ small enough.
    \item[(x)] The last assumption we need to check is the spectral gap condition \eqref{spectral gap condition}. To this end let $\varepsilon\zeta ^{-1}\leq \frac{1}{2}$ and assume that the parameter $\nu$ is chosen large enough such that $C_A a \tilde\gamma <\frac{1}{2} \nu$.
    Then,
    $C(L_{g_1}+L_{g_2})<\frac{1}{2}( N_S^\zeta-N_F^\zeta)$ holds for all sufficiently small $\zeta>0$ and the spectral gap condition is satisfied.
\end{itemize}
This completes checking the assumptions for section \ref{sect 4.1}.

\begin{remark}
    The dissipative structure of the system $\eqref{Beddington system}$ is key that the we do not loose to much information on the solutions when we apply the cut-off method to obtain global Lipschitz bounds.  
\end{remark}
\begin{remark}
    We want to emphasize that the slow manifold for the original system exist only in the region of the phase space that satisfies the uniform bound already prior to the modification and cut-off procedure.
    This means that in the application at hand, the exponential attraction towards the slow manifold (cf. Proposition \ref{Prop 4.9}) is only valid in the aforementioned region of phase space.
    Hence, to see the slow manifold in numerical simulations it is crucial for the initial data to be already close to the critical manifold.
\end{remark}

\subsection{Remarks}
We conclude this example with some observations.

\begin{remark}
 As pointed out in \cite{debussche1991inertial,temam2012infinite} slow manifolds can be seen as a spatial case of inertial manifolds. And systems such as (1.1) with a large reaction term can have inertial manifolds as shown in \cite{zelik2014inertial} and the references therein.
 However, the existence of the invariant manifold (inertial or slow) plays only one part in this work and the slow manifold literature.
 The key question, more importantly, is the convergence of the slow manifold $S_{\varepsilon,\zeta}$ to the critical manifold $S_0$ in both a geometric and analytical way.
\end{remark}

\begin{remark}
    Here, we want to explain the above mentioned convergence in more detail.
    The critical manifold $S_0\subset X\times Y$ given by $S_0=\{(x,y)\in X\times Y: f(u,v)=0\}$ is a (possible) infinite dimensional manifold in the phase space on which the slow dynamics evolve.
    With the splitting of $Y$ into a fast component $Y_F^\zeta$ and a slow component $Y_S^\zeta$ we obtain that the slow component has a finite dimension, where $\dim Y_S^\zeta$ is related to $\zeta^{-1}$.
    As the slow manifold is constructed as a function $(h_{X_1}^{\varepsilon,\zeta}(v_S),h^{\varepsilon,\zeta}_{X_F^\zeta}(v_S))$ over the slow variable space it follows that $\dim S_{\varepsilon,\zeta}\leq  2\dim Y_S^\zeta $.
    Thus, with decreasing $\zeta$ the dimension of the slow variable space and therefore also of the slow manifold increases and in the limit $\zeta \to 0$ coincides with the critical manifold.
\end{remark}

\begin{remark}
    In applications one might to need to construct the explicit slow manifold. However, as the current existence proof of the slow manifold is non-constructive the question is whether this is possible. One method to obtain an approximate slow manifolds is via a Galerkin approximation and the ideas presented in \cite{engel2021connecting} for fast-slow systems also hold in the case of fast-reaction equations.
\end{remark}

\begin{remark}
    As in the theory for inertial manifolds, the question whether there are slow manifolds in the case when the spectral gap condition is violated remains open. We conjecture that the spectral gap condition is only a sufficient condition for the existence of the slow manifold and that there are cases similar to the one in \cite{zelik2014inertial} that give us examples of existence but also counterexamples.
\end{remark}

\begin{remark}
    We note some differences between the results obtained for fast-reaction systems in this work and the results in \cite{hummel2022slow} for fast-slow systems.
    One key issue are the restrictions on the operator $A$. 
    In this work the growth bound on the semigroup generated by the operator $A$ need not satisfy $\omega_A<0$ since the fast-reaction dominates the terms on the right-hand side.
    Therefore, we only require that $\| \textnormal{e}^{\varepsilon^{-1}(\varepsilon A -\textnormal{D}_x f(h^0(v^0),v^0))t} \|_{\mathcal{B}(X_1)}\leq M_A \textnormal{e}^{\varepsilon^{-1}(\varepsilon \omega_A- \lambda_0)t} $ where $\varepsilon \omega_A- \lambda_0< 0$.
    On the other hand, we need more requirements on the nonlinear function $f$, one of which is, that $f$ has to satisfy the implicit function theorem in a neighborhood around the equilibrium $f(u,v)=0$. 
    Another difference is that, with the current method, we loose regularity in the estimates for the slow manifold, i.e. the distance result (Prop. \ref{Prop 4.6}) and the convergence results (Prop. \ref{Prop 4.9}) only hold in the space $X_\alpha\times Y_1$ and not in $X_1\times Y_1$.
\end{remark}

\begin{appendix}
\section{Tools from Functional Analysis}\label{appendix A}
In this section we present useful definitions and theorems from the theory of semigroups that are used throughout the paper.
The details to the brief summary we present here can be found in \cite{amann1995linear} and \cite{ lunardi2012analytic} and the references therein.

\begin{definition}
Let $X$ be a Banach space. 
Then a family of linear, bounded operators $\big( S(t)\big)_{t\geq 0}\subset \mathcal{L}(X)$, where $S(t): X\to X$ for all $t\geq 0$, is called a semigroup if
\begin{align*}
    S(0)= I,\qquad S(s+t)=S(s)S(t) \quad \forall t\geq 0.
\end{align*}
It is called a strongly continuous or $C_0$-semigroup if in addition
\begin{align*}
    \|S(t)x-x\|_X\to 0,\quad t\to 0^+ \quad \forall x\in X.
\end{align*}
\end{definition}

\begin{definition}
Define the set $D(A):= \{ x\in X :\, \exists \, \lim_{t\to 0^+} \frac{S(t)x-x}{t}\}$ and $Ax:= \lim_{t\to 0^+} \frac{S(t)x-x}{t}$ for $x\in D(A)$.
Then, $A$ is called the infinitesimal generator of $\big( S(t)\big)_{t\geq 0}$.
\end{definition}

\begin{theorem}
Let $x\in D(A)$ and assume that $\sup_{t\geq 0} \|S(t)\|<\infty$.
Then
\begin{itemize}
    \item[i)] $S(t)x \in D(A)$ for all $t\geq 0$.
    \item[ii)] $AS(t)x = S(t)A x$ for all $t\geq 0$.
    \item[iii)] The mapping $t\mapsto S(t)x$ is differentiable for each $t\geq 0$ and $\frac{\textnormal{d}}{\textnormal{d}t} S(t)x= AS(t)x$ for all $t\geq 0$.
\end{itemize}
\end{theorem}

\begin{theorem}
Let $\big( S(t)\big)_{t\geq 0}$ be a $C_0$-semigroup on $X$. 
Then there exists an $\omega\in \mathbb{R}$ and $M\geq 1$ such that
\begin{align*}
    \|S(t)\|_{\mathcal{B}(X)}\leq M \textnormal{e}^{\omega t}\quad \forall t\geq 0.
\end{align*}
\end{theorem}

\begin{theorem}
Let $A$ be the generator of a $C_0$-semigroup $\big( S(t)\big)_{t\geq 0}$ on $X$.
If $B\in \mathcal{B(X)}$, then $A+B$ is the infinitesimal generator a $C_0$-semigroup $\big( T(t)\big)_{t\geq 0}$ on $X$ satisfying
\begin{align*}
   \|T(t)\|_{\mathcal{B}(X)}\leq M \textnormal{e}^{(\omega +M\|B\|_{\mathcal{B}(X)})t}\quad \forall t\geq 0.
\end{align*}
\end{theorem}

\begin{theorem}
Let $A$ be the generator of a $C_0$-semigroup $\big( S(t)\big)_{t\geq 0}$ on $X$ and let $x\in D(A)$.
Then the mapping $t\mapsto S(t)x $ solves the abstract Cauchy problem on X for $t\geq 0$
\begin{align*}
    \frac{\textnormal{d}}{\textnormal{d}t} y=Ay, \quad y(0)=x.
\end{align*}
\end{theorem}

\begin{definition}
Let $A$ be a densely defined closed linear operator on $X$ with $0\in \varrho(A)$.
Let $\theta \in (0,1)$.
Then we call $(\cdot,\cdot)_\theta$ an exact admissible interpolation functor, i.e. it is an exact interpolation functor such that $X_1$ is dense in $(X_0,X_1)_\theta$ whenever $X_1 \overset{d}{\hookrightarrow}X_0$.
\end{definition}

\begin{definition}
We define a family of Banach spaces $(X_\alpha)_{\alpha\in [-1,\infty)}$ and a family of operators $(A^\alpha)_{\alpha\in [-1,\infty)}\in \mathcal{B}(X_\alpha,X_{\alpha+1})$ as follows
\begin{itemize}
    \item for $k\in \mathbb{N}_0$ we set $X_k:= D(A^k)$ endowed with the norm $\|x\|_{X_k}:= \|A^k x\|_{X}$.
    In particular, $X_0=D(A^0)=X$;
    \item $X_{-1}$ is defined as the completion of $X$ with respect to the norm $\|x\|_{X_{-1}}:= \|A^{-1}x\|_{X}$;
    \item For $k\in \mathbb{N}_0$, $\theta\in (0,1)$ and $\alpha=k+\theta$ we define
    $X_\alpha:= (X_k,X_{k+1})_\theta$ and $A^\alpha= A^k\big|_{D(A^\alpha)}$, where
    $D(A^\alpha)=\{ x\in X_{k+1}:\, A^k x \in X_\alpha\}$.
\end{itemize}
The family $(X_\alpha,A^\alpha)_{\alpha\in [-1,\infty)}$ is a densely injected Banach scale in the sense that $X_\alpha\overset{d}{\hookrightarrow}X_\beta$ whenever $\alpha\geq \beta$ and $A^\alpha: X_{\alpha+1} \to X_\alpha$ is an isomorphism for all $\alpha\in \mathbb{R}$.
Moreover, $A^\alpha: X_{\alpha+1}\to X_\alpha$ is a densely defined closed operator with $0\in \varrho(A^\alpha)$ for all $\alpha\in \mathbb{R}$.
The family $(X_\alpha,A^\alpha)$ is called an interpolation-extrapolation scale.
\end{definition}

\begin{theorem}
Let $A$ be the generator of a $C_0$-semigroup $\big( S(t)\big)_{t\geq 0}$ on $X$ and let $\omega_S\in \mathbb{R}$ be the growth bound of the semigroup $S(t)$.
Then $A^\alpha:X_{\alpha+1}\to X_\alpha$ also generates a $C_0$-semigroup $\big( S_\alpha(t)\big)_{t\geq 0}$ with the same growth bound. 
For all $\alpha,\, \beta \in [0,\infty)$ with $\alpha\geq \beta$ the following diagram commutes
\[\begin{tikzcd}
	{X_\alpha} & {X_\alpha} \\
	{X_\beta} & {X_\beta}
	\arrow[hook, from=1-1, to=2-1]
	\arrow[hook, from=1-2, to=2-2]
	\arrow["{S_\alpha(t)}", from=1-1, to=1-2]
	\arrow["{S_\beta(t)}", from=2-1, to=2-2]
\end{tikzcd}\]
and for all $\omega>\omega_S$ there is a constant $C$ depending on $\alpha$ and $\beta$ such that
\begin{align*}
   \|S_\beta(t)\|_{\mathcal{B}(X_\beta,X_\alpha)} \leq C t^{\beta-\alpha} \textnormal{e}^{\omega t}\quad \forall t>0.
\end{align*}

\end{theorem}

\section{A different approach in controlling the fast-reaction term} \label{Appendix B}

In this section we present an alternative approach in constructing the slow manifold.
We want to emphasize however that this is only a sketch of the approach and that the full development of this approach needs yet to be done.

The idea to the this approach comes form studying a finite dimensional fast-slow ODE.
To this end, we study the following planar system
\begin{align}
    \begin{split}
        \frac{\textnormal{d}}{\textnormal{d} t}\begin{pmatrix}u^\varepsilon\\v^\varepsilon \end{pmatrix}= \begin{pmatrix} -k^2 &0\\0&-1 \end{pmatrix} \begin{pmatrix} u^\varepsilon\\v^\varepsilon\end{pmatrix} +\begin{pmatrix} \frac{1}{\varepsilon} f(u^\varepsilon,v^\varepsilon)\\ 0\end{pmatrix},
    \end{split}
\end{align}
where $k > 0$ is a fixed constant. In this example let $f(u,v)= -u+v^2$.
Then, the ODE system has the explicit solution
\begin{align}
u^\varepsilon(t)=u_0 \textnormal{e}^{-t(k^2+\varepsilon^{-1})} + \frac{(v_0)^2 \textnormal{e}^{-2t}} {(-2+k^2+\varepsilon^{-1})\varepsilon}, \qquad v^\varepsilon(t)=v_0 \textnormal{e}^{-t}.
\end{align}
The limit system as $\varepsilon\to 0$, which reads
\begin{align}
     \frac{\textnormal{d}}{\textnormal{d} t} v^0=-v^0, 
\end{align}
defined on the critical manifold $S_0=\{(u,v)\in \mathbb{R}^2: u=v^2\}$, has the solution $v^0=v_0 \textnormal{e}^{-t}$.\\

Then, we can observe a different behavior of the solution in $u$ for different ratios between $k^2$ and $\varepsilon^{-1}$.
Assuming that $k^2\ll \varepsilon^{-1}$ and the initial condition satisfy $u_0=(v_0)^2$, then, the fast-reaction term is dominating and we obtain that $u^\varepsilon(t)= v^\varepsilon(t)^2 +\mathcal{O}(\varepsilon)$.
This implies that in this regime for $t>0$ we have
\begin{align*}
    f(u^\varepsilon(t),v^\varepsilon(t))&= (v_0)^2\bigg(\textnormal{e}^{-t(k^2+\varepsilon^{-1})} - \frac{(-2\varepsilon+k^2\varepsilon)\textnormal{e}^{-2t}} {-2\varepsilon+k^2\varepsilon+1} \bigg)\\
    &\leq \varepsilon C (v_0)^2 \textnormal{e}^{-2t} = \varepsilon C h^0(v^0(t)),
\end{align*}
where $h^0(v)= v^2$.This estimate will the be key to compensate the $\varepsilon^{-1}$-term coming from the fast-reaction.

In the other regime, when $k^2\geq \varepsilon^{-1}$, we observe that the linear (diffusion) term dominates and we have that $u^\varepsilon(t)$ decays to zero much faster than $v^\varepsilon(t)$ as $t$ increases.\\

The next step is to generalize this idea to the PDE setting, where we consider a simplified fast-reaction equation of the form
\begin{align}\label{appendix problem}
\begin{split}
    \partial_t u^\varepsilon &= A u^\varepsilon +\frac{1}{\varepsilon} f(u^\varepsilon, v^\varepsilon),\\
    u^\varepsilon(0)&= u_0,
    \end{split}
\end{align}
where $v^\varepsilon\in \mathbb{R}$ is a parameter depending on $\varepsilon$.
The corresponding limit problem is 
\begin{align}
    0= f(u^0,v^0).
\end{align}
In order for this problem to be well-defined, we assume that $f$ satisfies the assumptions of the implicit function theorem, such that there exists a function $h^0:\mathbb{R}\to X_1$ with $f(h^0(v^0),v^0)=0$. As before we further assume that the spectrum of $\textnormal{D}_x f(x,y)$ is located on the left half-plane.\\

Motivated by the finite-dimensional ODE example we introduce a splitting of the Banach space $X$ of the form
\begin{align*}
    X=X^\xi_F \oplus X^\xi_S
\end{align*}
into a fast part $X^\xi_F$, i.e. the case where the linear operator dominates the effect of the nonlinearity, and a slow part $X^\xi_S$, i.e. the case where the fast reaction term drives the flow of the solution.
Here $\xi>0$ is a small parameter encoding the splitting of the Banach space and we assume that the following holds
 \begin{itemize}
    \item The spaces $X^\xi_F$ and $X^\xi_S$ are closed in $X$ and the projections $\pr_{X^\xi_F}$ and $\pr_{X^\xi_S}$ commute with $A$ on $X_1$.
    \item The spaces $X^\xi_S\cap X_1$ and $X^\xi_F\cap X_1$ are closed subspaces in $X$ and will be endowed with the norm $\|\cdot\|_{X_1}$.
    \item The realization of the operator $A$ in $X^\xi_S$ is given by 
         $A_{X^\xi_S}: D(X^\xi_S)\subset X^\xi_S \to X^\xi_S$
    with 
    $$D(A_{X^\xi_S})=\{u\in X^\xi_S\cap D(A):\, Au\in X^\xi_S\}$$
    and similar for the realization of $A$ in $X^\xi_F$.
 \end{itemize}
 In addition, we make the following assumptions on the nonlinearity and the splitting based on the previous observations.
 \begin{itemize}
      \item The space $X^\xi_F$ contains the parts of $X_1$ that decay very fast under the semigroup $(\textnormal{e}^{tA})_{t\geq 0}$ satisfying the following estimate
    \begin{align*}
        \|\textnormal{e}^{tA} x_F\|_{X_1}\leq M_A \textnormal{e}^{\varepsilon^{-1} \omega_A t} \|x_F\|_{X_1},
    \end{align*}
  where we observe a much faster decay when compared with the general estimate.
    \item The nonlinearity $f$ satisfies
    \begin{align*}
        \|\pr_{X^\xi_S} [f(u_1,v_1)-f( u_2,v_2)]\|_{X_1}\leq   \varepsilon L_f\big(\|u_1-u_2\|_{X_1}+ |v_1-v_2|\big),
    \end{align*}
    where $\pr_{X^\xi_S}$ denotes the projection onto $X^\xi_S$ and $u_1,u_2\in X_1$ are solutions to \eqref{appendix problem} with parameters $v_1$ and $v_2$ respectively.
    \item We require the initial data under this splitting to satisfy 
    \begin{align*}
         \pr_{X^\xi_S} [u_0-h^0(v^0)]=0.
    \end{align*}
 \end{itemize}
Next, we show how to apply this method to equation \eqref{appendix problem}, when studying the convergence of $u^\varepsilon \to h^0(v^0)$ in $X_1$ as $\varepsilon \to 0$.
To simplify the computations, we assume that, as before in this section, the nonlinearity $f$ is given by $f(u,v)= -u+v^2$.
Then, setting $w^\varepsilon=u^\varepsilon-h^0(v^0)$ we obtain 
\begin{align*}
    \|w^\varepsilon(t)\|_{X_1} &= \bigg\|  \textnormal{e}^{At} w_0 +\varepsilon^{-1} \int_0^t \textnormal{e}^{A(t-s)} f(w^\varepsilon(s)+h^0(v^0),v^\varepsilon)\,\textnormal{d}s  \bigg\|_{X_1}\\
    &\leq M_A \int_0^t \bigg(\textnormal{e}^{\omega_A (t-s)} +\varepsilon^{-1} \textnormal{e}^{\varepsilon^{-1}\omega_A (t-s)}\bigg) \big( \|w^\varepsilon(s)\|_{X_1}+|h^0(v^0)-(v^\varepsilon)^2|\big)\,\textnormal{d}s \\
    &\quad + M_A \textnormal{e}^{\varepsilon^{-1}\omega_A t} \| \pr_{X^\xi_F} w_0\|_{X_1},
\end{align*}
where we used the splitting of the space $X$ to control the $\varepsilon^{-1}$-terms. Noting that $h^0(v^0)=(v^0)^2$ we obtain 
\begin{align*}
     \|w^\varepsilon(t)\|_{X_1} &\leq M_A \textnormal{e}^{\varepsilon^{-1}\omega_A t} \| \pr_{X^\xi_F} w_0\|_{X_1}+ M_A |\omega_A^{-1}| \big|(v^0)^2-(v^\varepsilon)^2\big| \\
     &\quad + M_A \int_0^t \bigg(\textnormal{e}^{\omega_A (t-s)} +\varepsilon^{-1} \textnormal{e}^{\varepsilon^{-1}\omega_A (t-s)}\bigg)  \|w^\varepsilon(s)\|_{X_1}\,\textnormal{d}s.
\end{align*}
This then yields
\begin{align*}
    \|w^\varepsilon(t)\|_{X_1} \leq C \big(\textnormal{e}^{\varepsilon^{-1}\omega_A t} \| \pr_{X^\xi_F} w_0\|_{X_1}+  \big|(v^0)^2-(v^\varepsilon)^2\big|\big)
\end{align*}
As the parameter $v^\varepsilon\to v^0$ as $\varepsilon\to 0$ we have $w^\varepsilon \to 0$, which shows the convergence of solutions.

\begin{remark}
    We want to point out that in this example we did not have to assume any global Lipschitz assumptions on the nonlinearity. However to compensate this lack of global control, more information on the specific problem at hand is required. 
    This reason makes the application of this method to general problems more difficult and we leave this to future projects to investigate.
\end{remark}

% \section{Versions of Gronwall's Inequality}\label{Appendix B}

% The following two versions of Gronwall's inequality are used throughout the paper and as a reference we point out \cite{dragomir2003applications} and \cite[Section 2.3]{hummel2022slow}.

% \begin{lemma}
% Let $T>0$ and let $u,\, v,\, c :[0,T]\to [0,\infty)$ be continuous and suppose that $c'$ is locally integrable. 
% If
% \begin{align*}
%     v(t)\leq c(t) +\int_0^t u(s)v(s)\, \textnormal{d}s\quad \forall t\in [0,T]
% \end{align*}
% then
% \begin{align*}
%     v(t)\leq c(0) \exp{\bigg(\int_0^t u(s)  \,\textnormal{d}s \bigg)} + \int_0^t c'(s)\exp{\bigg(\int_s^t u(r) \,\textnormal{d}r\bigg)}\, \textnormal{d}s.
% \end{align*}
% \end{lemma}

% \begin{lemma}
% Let $x\in \mathbb{R}$, $\varepsilon,\, N,\,T>0$, $\gamma\in (0,1]$ and let $p\in (1,\infty)$ with $p'=p/(p-1)$ being the conjugated index.
% Let  $v,\, c :[0,T]\to [0,\infty)$ be continuous.
% Suppose that $c'$ is locally integrable and that the mapping $t\mapsto \textnormal{e}^{-\varepsilon^{-1}xt} c(t)$ is non-decreasing.
% If 
% \begin{align*}
%     v(t)\leq c(t) +N \int_0^t \frac{\textnormal{e}^{-\varepsilon^{-1}x(t-s)}}{\varepsilon^\gamma (t-s)^{1-\gamma}} v(s)\, \textnormal{d}s\quad \forall t\in [0,T]
% \end{align*}
% then
% \begin{align*}
%     v(t)\leq p c(0) \textnormal{e}^{\varepsilon^{-1}\tilde xt} + p\int_0^t \big( c'(s)-\varepsilon^{-1}x c(s)\big) \textnormal{e}^{\varepsilon^{-1}\tilde x(t-s)}\, \textnormal{d}s,
% \end{align*}
% where $\tilde x= x+ pN^{1/\gamma} \big(\frac{p'}{\gamma}\big)^{(1-\gamma)/\gamma}$.
% \end{lemma}

\end{appendix}

\section*{Acknowledgements}
The authors would like to thank F. Hummel, B.Q. Tang and B.-N. Tran for the fruitful discussions. 
Furthermore, the authors acknowledge the support of the DFG under grant No.$456754695$.
\printbibliography
\end{document}